\documentclass[12pt]{article}
\usepackage{booktabs}
\usepackage[margin=2cm]{geometry}
\usepackage{setspace}
\usepackage{amsmath}
\usepackage{amssymb}
\usepackage{graphicx}
\usepackage{algorithm}
\usepackage{algpseudocode}
\usepackage{float}
\usepackage{subcaption}
\usepackage{epstopdf}
\usepackage{chngcntr}
\counterwithin{figure}{section}

\usepackage{amsthm}
\newtheorem{definition}{Definition}[section]
\newtheorem{theorem}{Theorem}[section]

\newtheorem{lemma}[theorem]{Lemma}
\newtheorem{corollary}[theorem]{Corollary}
\newtheorem{rem}[theorem]{Remark}

\newcommand{\REQUIRE}{\State{\textbf{Input:\ }}}
\newcommand{\ENSURE}{\State{\textbf{Output:\ }}}

\newcommand{\floor} [1] {\left \lfloor{#1}\right \rfloor}

\renewcommand{\eqref}[1]{\text{Eq.~}\ref{#1}}

\usepackage{color}
\usepackage{xparse}
\newcommand{\Nystrom}{Nystr\"{o}m~}

\newcommand{\rvc}[2]{
\textcolor{red}{#1}%
    \IfNoValueF{#2}{\protect\footnotemark{\footnotetext{\textcolor{red}{#2}}}%
    }%
}

\begin{document}

\title{Fast and Accurate Gaussian Kernel Ridge Regression Using Matrix Decompositions for Preconditioning}
\author
{Gil Shabat\footnote{Corresponding author: gils@playtika.com}~~~~Era Choshen~~~Dvir Ben Or~~~Nadav Carmel\\
Playtika AI Research Lab, Israel\\
}
\date{}
\maketitle

\begin{abstract}
	This paper presents a preconditioner-based method for solving a kernel ridge regression problem. In contrast to other methods, which utilize either fast matrix-vector multiplication or a preconditioner, the suggested approach uses randomized matrix decompositions for building a preconditioner with a special structure, that can also utilize fast matrix-vector multiplications. This hybrid approach is efficient in reducing the condition number, exact and computationally efficient, enabling to process large datasets with computational complexity linear to the number of data-points.  Also, a theoretical upper bound for the condition number is provided. For Gaussian kernels, we show that given a desired condition number, the rank of the needed preconditioner, can be determined directly from the dataset.
\end{abstract}


\section{Introduction}
\label{sec:intro}
Kernel methods are a way to embed the input space into a high dimensional feature space, which enables learning over a richer and more complex hypothesis class. One of the most elementary applications of kernel methods is Kernel Ridge Regression (KRR). The problem setup can be defined in the following manner - let the training data be defined as $(x_1, y_1), \dots, (x_n, y_n) \in \mathcal{X} \times{\mathcal{Y}}$  where $\mathcal{X} \subseteq \mathbb{R}^{d} $ is the input domain and $ \mathcal{Y} \subseteq \mathbb{R} $ is the output domain, let $k: \mathcal{X} \times \mathcal{X} \rightarrow \mathbb{R}$, and a regularization parameter $\beta > 0$, with the response for a given input x estimated as:
\begin{equation}
f(x) = \sum_{j=1}^{n}{k(x_j,x)\alpha_j}
\end{equation}
where $\alpha=(\alpha_1\cdots \alpha_n)^{T}$ is the solution of the equation
\begin{equation}
\label{eq:ridge_regression}
(K+\beta I)\alpha = y
\end{equation}
and $y=(y_1 \cdots y_n)^{T}$ is the label vector.
\begin{rem}
$K \in \mathbb{R}^{n\times n}$ is the kernel matrix defined by $K_{ij}=k(x_i,x_j)$
\end{rem}
KRR, while simple, is a powerful technique that achieves empirically good results. However, KRR comes with the drawback of having to compute the KRR estimator. The straightforward solution for the estimator requires inverting the matrix, which requires $\mathcal{O}(n^3)$ time, and is very restrictive for large datasets. Furthermore, holding the original kernel in memory requires $\mathcal{O}(n^2)$ storage, which by itself is prohibitive for large n.

These drawbacks have driven research to focus on scaling up kernel based methods, mostly by approximation, yielding methods such as the Nystr\"{o}m method for low-rank kernel approximation, or the randomized construction of feature maps such as Random Fourier Features (RFF), originally proposed by Rahimi and Recht \cite{rahimi_recht}.

Another approach for the kernel ridge regression problem is using a preconditioner \cite{shen2006fast, freitas2006fast}. It was recently suggested to use RFF as a preconditioner for the kernel matrix \cite{avron_krr} and then to use conjugate gradient to solve the regression problem. In this paper, a different approach for the preconditioner is presented, based on randomized matrix factorizations \cite{halko2011finding,mahoney2011randomized,shabat2018randomized,clarkson2017low}. The suggested preconditioner combines between the interpolative decomposition, \Nystrom approximation and fast improved Gauss transform to achieve an efficient preconditioner that can be applied in $\mathcal{O}(n)$ operations, i.e. linear in the number of data points.

\subsection{Related work}
Building scalable and computationally efficient algorithms for kernel methods is an active research area. Some methods use low rank approximations to approximate the kernel directly, such as Nystr\"{o}m approximation or random Fourier features. This approach is very popular and has many variants where the main idea is to reduce the computational cost by low rank approximation while retaining reasonable model quality. 
Recent work on these approaches can be found in \cite{lu2014scale, dai2014scalable,avron2016high,musco2017recursive,rudi2017falkon,wang2019block,williams2001using,smola2000sparse}. Though these approaches are efficient, both computationally and performance (accuracy) wise, the obtained solution is an approximation. The quality of the approximation is data dependent and therefore there can be scenarios  where the number of samples needed to achieve satisfying results can be relatively large, for example when the data is uniformly spread all over its support.
A different approach uses fast matrix vector multiplications (MVM), e.g fast Gauss transform or fast multipole methods to solve the regression problem using Krylov iterations \cite{raykar2007fast,freitas2006fast,shen2006fast,alfke2018nfft}. 
These methods reduce the computational cost of each Krylov iteration significantly, from $\mathcal{O}(n^2)$ to $\mathcal{O}(n\log n)$ or even $\mathcal{O}(n)$. However, when the kernel matrix is ill-conditioned, the number of iterations required might be huge.
Other approaches, discussed in \cite{avron_krr,srinivasan2014preconditioned,cutajar2016preconditioning,ma2017diving}, use a preconditioner to reduce the condition number, which in turn reduces the required number of Krylov iterations.
Some methods combine between the ideas mentioned above, for example by using a preconditioner along with fast MVM. The method presented in \cite{gardner2018gpytorch,wang2019exact} uses a truncated pivoted Cholesky decomposition and L-BFGS optimizer with a fast GPU-based matrix-vector multiplication which yields fast and scalable algorithm. 

\subsection{Contributions}
There is usually a tradeoff between using a fast MVM kernel matrix approach and applying a preconditioner. The fast MVM is made possible to due to the special structure of the kernel, however most preconditioners do not maintain this special structure required for fast MVM, incurring an additional computational cost in their application. 
This paper suggests a solution to this tradeoff by creating a preconditioner that can be applied with fast MVM by maintaining the kernel matrix’s unique structure.  This preconditioner essentially enjoys both worlds, it reduces the condition number but does not incur the cost of performing regular matrix multiplication instead of fast MVM. This yields an algorithm that not only utilizes fast MVM on the kernel matrix, but also has a reduced condition number due to a preconditioner that can itself be applied using fast MVM.
Given $n$ the number of data points, $l \ll n$ a sampling parameter and $t$ the number of iterations, the suggested algorithm has a runtime complexity of $\mathcal{O}(nl^2+nt)$ and memory consumption of $\mathcal{O}(nl)$.
The algorithm can provide a solution with any desired accuracy, up to the machine precision where there is a tradeoff between $l$ and $t$ - selecting more sampling points (larger $l$) reduces the condition number yielding a smaller number of iterations $t$.


In addition, we present a theoretical analysis of the condition number of the preconditioned matrix and give a theorem for its upper bound. 
Furthermore, for a Gaussian kernel, the paper provides a theoretical lower bound for the low rank approximation required to reach any desired condition number. This lower bound can be calculated directly from the data, without building the kernel matrix or conducting any other complex computations.
\section{Preliminaries}
\subsection{Notation}
Throughout the paper, matrices are indicated by capital letters and vectors and scalars by small letters. 
The norm $\Vert\cdot\Vert$ when applied to vectors, refers to the standard Euclidean norm, i.e. $\Vert v \Vert = \sqrt{\sum_{i=1}^n{v_i^2}}$ and to the spectral norm when applied to matrices, i.e. $\Vert A \Vert = \sup_{v \neq 0} \frac{\Vert Av\Vert}{\Vert v \Vert}$. The norm $\Vert \cdot \Vert_A$ applied to a vector with respect to a positive semidefinite matrix $A$ is defined $\Vert v \Vert_A=\sqrt{\langle v,Av \rangle}$. Singular values are indicated in descending order, by $\sigma_1,\ldots,\sigma_n$ and eigenvalues by $\lambda_1,\ldots,\lambda_n$ also in descending order when they are real. $A \ge 0$ on a square matrix indicates that $A$ is positive semidefinite and $A \ge B$ means $A-B \ge 0$. $A^\dagger$ is used to indicate the pseudo inverse of $A$.
\subsection{\Nystrom Approximation}
The \Nystrom approximation is an important low rank matrix approximation method which can applied to symmetric positive semidefinite (SPSD) matrices. The method obtains a low rank approximation to a kernel matrix by using a subset of its columns. Formally, suppose $K \in \mathbb{R}^{n\times n}$ is a symmetric positive semidefinite matrix, $C \in \mathbb{R}^{n\times k}$ a subset of the columns of $K$ and $U \in \mathbb{R}^{k\times k}$ is the overlap between $C$ and $C^T$, then $K$ can be approximated as:
\begin{equation*}
K \approx CU^{\dagger}C^T.
\end{equation*}
The quality of the approximation depends on the selection of the columns. Though the factorization seems identical to a special case of SPSD matrix approximated by
 the CUR approximation (described below), the theory of the \Nystrom method originates from the discretization of eigenfunctions of continuous operators. 
 For more information the reader is referred to \cite{nemtsov2016matrix,williams2001using,drineas2005nystrom}. 
 In this paper, ``Nystr\"{o}m" is used when the approximation is applied to an SPSD matrix, and CUR to a general matrix.
\subsection{Interpolative and CUR decompositions}
\label{sec:id_cur}
Interpolative decomposition (ID) \cite{cheng2005compression} is a matrix factorization that can be applied to any matrix $A$ and defined as the following:
\begin{equation}
\label{eq:ID1}
A_{m\times n} \approx B_{m \times k}T_{k\times n}
\end{equation}
where $B$ is a subset of the $k$-columns of $A$ and $T$ is an interpolation matrix with certain properties (such as spectral norm bounded by a factor dependent on the  size of the matrix \cite{liberty2007randomized}) that is used for reconstructing $A$ from $B$. 
The columns of $B$ are computed by rank revealing factorization \cite{gu1996efficient,pan2000existence,miranian2003strong,chandrasekaran1994rank,RRQR_svd} in order to get an error $\Vert A-BT\Vert$ bounded by a factor proportional to the $k$th singular value of $A$. 
In more detail, using the development from \cite{cheng2005compression}, suppose $AP=QR$ is a strong rank revealing QR (RRQR) factorization of $A$ (suppose $m \ge n$), where $P$ is an $n \times n$ permutation matrix, $Q$ is an $m \times n$ matrix with orthonormal columns and $R$ is an $n \times n$ upper triangular matrix. Splitting $Q$ and $R$:
\begin{equation*}
Q = \begin{bmatrix}
Q_{11} & Q_{12} \\
Q_{21} & Q_{22}
\end{bmatrix}~~~~~\mbox{and}~~~~R=\begin{bmatrix}
R_{11} & R_{12} \\
0 & {R_{22}}
\end{bmatrix}
\end{equation*}
such that $Q_{11}$ and $R_{11}$ are of size $k \times k$, $Q_{21}$ is $(m-k) \times k$, $Q_{12}$ is $k\times (n-k)$, $Q_{22}$ is $(m-k) \times (n-k)$,  and $R_{12}$ is $k \times (n-k)$ and $R_{22}$ is $(n-k)\times (n-k)$ then
\begin{equation}
\Vert AP - \begin{bmatrix} Q_{11} \\ Q_{21} \end{bmatrix} \begin{bmatrix} R_{11} & R_{12} \end{bmatrix} \Vert_2 \le \sqrt{1+k(n-k)}\sigma_{k+1}(A).
\label{eq:boundid}
\end{equation}
By defining:
\begin{equation*}
B \triangleq  \begin{bmatrix} Q_{11}R_{11} \\ Q_{21}R_{11} \end{bmatrix}
\end{equation*}

Note that the definition of $B$ reconstructs the first $k$ columns of $AP$ from the matrices $Q$ and $R$. Therefore, since $B$ is determined strictly by the permutation matrix $P$ and $T$ is not used in the algorithm, it is sufficient only to apply an RRQR factorization for the subset selection.
In addition to the deterministic algorithm for computing the ID that is described in \cite{cheng2005compression}, there is a randomized version that starts by projecting $A$ onto a low dimensional space using a random normally distributed matrix \cite{martinsson2011randomized,halko2011finding}. The projection step of this randomized version is used in Algorithm \ref{alg:ID_Sampling}.

The CUR decomposition \cite{mahoney2009cur,drineas2008relative}, is a pseudo-skeleton decomposition \cite{goreinov1997theory, osinsky2018pseudo} that factors a matrix $A$ into three matrices, where $C$ and $R$ are subsets of the columns and rows of $A$ respectively, and $U$ is defined as the inverse matrix of the overlap between $C$ and $R$. 
Most of the matrices discussed in the paper are symmetric, so the Nystr\"{o}m approximation will be used as a special case of the CUR decomposition. The following Lemma gives an error bound for the CUR decomposition using columns and rows selected by the interpolative decomposition:
\begin{lemma}\cite{voronin2017efficient}
	\label{lemma:curid}
	Let $A \in \mathbb{R}^{m\times n}$, that satisfies $A= CV^{T} + E$ and $A=WR+\tilde{E}$, where C,R are the k columns and k rows of A, respectively, and W,$V^{T}$ are the interpolation matrices from the interpolative decomposition. Suppose further that R is full rank, and that U $\in \mathbb{R}^{kxk}$ is defined as $U=V^TR^{\dagger}$. Then
	\begin{equation}
	\vert \vert A - CUR \vert \vert \le \vert \vert E \vert \vert + \vert \vert \tilde{E} \vert \vert.
	\end{equation}
	\begin{rem}
		Note that $ \Vert E \Vert \le p(n,k)\sigma_{k+1}$ \cite{cheng2005compression}. Moreover, the bound in Eq. \ref{eq:boundid} is optimal. Practical algorithms for computing the  deterministic ID achieve $p(n,k)= \sqrt{1+4k(n-k)}$ \cite{liberty2007randomized}. Equivalently, a similar bound exists for the randomized ID, the reader is referred to \cite{martinsson2011randomized}.
	\end{rem}
\end{lemma}
\subsection{Fast Gauss Transform}
Fast Gauss transform (FGT) is a variant of the fast multipole method (FMM) \cite{greengard1987fast}. FMM was originally formulated as an efficient method for complex physical calculations, efficiently performing matrix-vector products. FGT deals with the evaluation of the discrete Gauss transform:
\begin{equation}
\label{eq:dscrt_gt}
G(y_j) = \sum_{i=1}^N q_{i}e^{-\Vert y_j - x_i \Vert ^{2} / h^2}
\end{equation}	
where $\{x_i\}_{i=1}^N, x_i \in \mathbb{R}^d$ represent the centers of the Gaussians, each of which has bandwidth $h$, and $q_i$ are the weight coefficients for each Gaussian. Direct evaluation of the summation for a set of points, $\{y_j\}_{j=1}^M, y_i \in \mathbb{R}^d$, is computationally costly in large scale scenarios, as it requires $\mathcal{O}(MN)$ operations. FGT allows setting a desired degree of precision $\Delta$ for the approximation of the Gaussian function, 
and reduces the complexity of the summation (\ref{eq:dscrt_gt}) to $\mathcal{O}(M + N)$, accompanied by a constant factor, which grows exponentially with the dimension $d$
 \cite{greengard1991fast}. In order to overcome this disadvantage, which makes plain FGT inapplicable to high-dimensional data, 
 the Fast Improved Gauss (FIG) transform \cite{yangimproved} uses tree data structures and an adaptive space subdivision technique, leading to reduction of the constant factor to asymptotically polynomial order, 
 such that the computational complexity is $\mathcal{O}\left((M + N)\text{poly}_p(d)\right)$, where $p$ is the degree of the polynomial, that depends on the desired accuracy. 
 This polynomial dependency, though still high, scales much better in the dimension comparing to exponential growth. It is worth noting, that the computational complexity in $d$ is dependent on the data as well as the selection of $h$  \cite{morariu2009automatic}.
\begin{definition}
	\label{def:figtree_def}
	Let $FIG(X,Y,h,q)$ represent the application of the FIG transform to compute equation (\ref{eq:dscrt_gt}), where $X \in \mathbb{R}^{N \times d}$ and $Y \in \mathbb{R}^{M \times d}$ represent matrices of data points for which the kernel is generated, weighted by the vector $q \in \mathbb{R}^N$. According to this notation, for $M=N$ it holds that $FIG(Y,X,h,q)=FIG^*(X,Y,h,q)$.
\end{definition}
The algorithm discussed in this paper can be applied to any other kernel that can utilize fast matrix-vector product such as other FMM functions. An approach for fast matrix-vector product for any kernel and based on nearest neighbors is described in \cite{shen2006fast}. Another approach, which is based on non-uniform Fourier transform (NFFT), 
is presented in \cite{alfke2018nfft}. This approach, which is applicable to low and medium dimensions, has several advantages over the FIG transform: 
It is robust and can applied easily to kernels other than the Gaussian kernels, and it reaches higher accuracy under the same computational budget. 
Since the algorithm presented here uses MVM as a black box, the FIG can be easily replaced by other MVM, such as \cite{shen2006fast,alfke2018nfft} when it can give better performance (for example, kernel other than Gaussian or low to medium dimension).
 Most of these methods, from nearest neighbors to subdivision schemes, though different, have a lot in common since they are all
  based on the fact that the contribution of far electric charge (or mass) points to a field can be considered as a single far charge point (or mass) of their superposition.
\section{Description of the algorithm}
The proposed algorithm is based on conjugate gradient with a preconditioner designed not only to be effective in its ability to reduce the condition number, but also in its efficiency to be applied in terms of computational cost and memory consumption.
In order to keep the time and storage complexity to a minimum , the proposed scheme uses the Fast Improved Gauss transform \cite{yangimproved,yang2005efficient} (``FIG transform"). If the kernel is Gaussian, then the FIG transform can be used for applying the kernel to any vector in $\mathcal{O}(n)$, but the preconditioner has to be applied using the FIG transform as well, to keep the advantage of using a Gaussian kernel.
In order to achieve this, the kernel is approximated using a Nystr\"{o}m decomposition, i.e. $\tilde{K}=CUC^T$, where $C \in \mathbb{R}^{n \times k}$ is a subset of columns of $K$ and $U \in \mathbb{R}^{k \times k}$ is the inverse of a submatrix of $C$. Adding the ridge regression parameter $\beta$, the preconditioner $\tilde{K}+ \beta I$ is obtained and can be applied efficiently using the Woodbury 
matrix identity:
\begin{equation}
\begin{split}
\label{eq:woodbury}
(\tilde{K}+\beta I)^{-1}&=(CUC^T+\beta I)^{-1}\\&=\beta^{-1}I-\beta^{-1}C(\beta U^{-1}+C^TC)^{-1}C^T
\end{split}
\end{equation}
Since $C$ is a submatrix of $K$, then the application of both $C$ and $C^T$ to a vector, can be done using FIG transform in $\mathcal{O}(n+k)$. $U^{-1}$ is also a subset of $K$ (unlike $U$ itself, which involves matrix inversion) and can therefore be applied in the same way, in $\mathcal{O}(k)$. The Woodbury matrix identity requires computing the inverse of $\beta U^{-1}+C^TC$, however this is a small matrix of size $k\times k$.
Constructing the matrix $\beta U^{-1}+C^TC$ can be done by the application of the FIG transform to the identity matrix, $I$ (of size $k$) in $\mathcal{O}(nk)$ operations. Each application of Eq. \ref{eq:woodbury} involves solving the following linear system
\begin{equation}
\label{eq:to_factor1}
(\beta U^{-1}+C^TC)x=w,
\end{equation}
for the application of $(\beta U^{-1}+C^TC)$ to a vector $w$, so it is better to store the Cholesky factor of the matrix in Eq. \ref{eq:to_factor1}. However,  this matrix is sometimes ill-conditioned, for example, when taking many anchor points, some of them might be very close to each other, leading to numerical instability. To overcome this, Eq. \ref{eq:to_factor1} is modified to:
\begin{equation}
\label{eq:to_factor2}
(\beta I+U^{\frac{1}{2}}C^TCU^{\frac{1}{2}})U^{-\frac{1}{2}}x=U^{\frac{1}{2}}w.
\end{equation}
where the matrix $\beta I+U^{\frac{1}{2}}C^TCU^{\frac{1}{2}}$ is generally more stable, and its Cholesky decomposition can be applied to solve Eq. \ref{eq:to_factor2} for $z=U^{-\frac{1}{2}}x$ and then restoring $x$ using $x=U^{\frac{1}{2}}z$.  
\begin{equation}
\label{eq:chol_factor}
L^TL=\beta I+U^{\frac{1}{2}}C^TCU^{\frac{1}{2}}
\end{equation}
where $L$ is the Cholesky factor of size $k \times k$. 
The algorithm can be viewed as two steps:
\begin{itemize}
	\item Preparation stage (or ``setup" stage), which selects anchor data points from the dataset and also performs some computations to be used later in the iterative stage, such as building the matrix $L$ of the Cholesky decomposition and the matrix $U^{\frac{1}{2}}$.
	\item Iterative stage, which applies conjugate gradient using the preconditioner that was computed in the preparation stage.
\end{itemize}
Selection of the anchor points can be done in various ways, such as random sampling or farthest point sampling (FPS) \cite{bronstein2008numerical}. In this work, the suggested method is based on the interpolative decomposition (or equivalently, on  rank revealing decompositions, \cite{gu1996efficient,pan2000existence,miranian2003strong}), which has the following advantages:
\begin{itemize}
	\item It is strongly related to the CUR decomposition \cite{voronin2017efficient}, and therefore also to the Nystr\"{o}m approximation. The interpolative decomposition is also integral to the method, yielding theoretical spectral bounds that can give theoretical guarantees for the overall performance of the algorithm.
	\item It selects columns of the input matrix, which enables the usage of the FIG-transform later on in the algorithm. Selection of columns is equivalent to selection of points, since each column is based on the distance of a specific points to all the points, so selecting the $i$th column is equivalent to selecting the $i$th datapoint.
	\item The selection of the columns (points) is done in a way that yields a good low rank approximation to the original kernel matrix, and therefore allows to obtain a good preconditioner. 
	\item It has a randomized version, which can be computed (using the FIG-transform) in linear time and the most demanding computational part can be parallelized.
	\item The randomized version can be further accelerated by using sparse random projection \cite{clarkson2017low,aizenbud2016randomized,dasgupta2010sparse} where the non zeros entries are only in $\{ 1, -1\}$. This yields a much faster algorithm (i.e. Algorithm \ref{alg:sparse_ID_Sampling}) at the expense of performance.
\end{itemize}
Algorithm \ref{alg:ID_Sampling} uses normal distribution for the random projection, while Algorithm \ref{alg:sparse_ID_Sampling} uses sparse projections. Algorithm \ref{alg:BuildChol} builds the Cholesky factor and Algorithm \ref{alg:ApplyPrecond} uses this factor for the application of the algorithm. The main function, that uses those algorithms is described in Algorithm \ref{alg:SolveKRR}.
\begin{algorithm}[htp]
	\caption{\textsf{AnchorSelection}: Select anchor points using randomized interpolative decomposition approach (RRQR)}
	\label{alg:ID_Sampling}
	\begin{algorithmic}[1]
		\REQUIRE $X$ - Dataset, matrix of size $n \times d$.\\
		$k$ - Number points to choose (rank), \\
		$l \ge k$ - number of random projections, typically slightly larger than $k$\\
		$h>0$ - Width of the Gaussian kernel
		\ENSURE $S$ - A set of anchor points from $X$
		\State Generate a random matrix $\Omega \in \mathbb{R}^{l\times n}$, s.t. $\omega_{i,j} \sim \mathcal{N}(0,1)$
		\State Y $\leftarrow$ FIG($X$,$X$, $h$, $\Omega$) \# Apply FIG transform to $\Omega$
		\State $Y^TP=QR$ \# Apply Strong RRQR to $Y$, $P$ is a permutation matrix, but can viewed as a vector
		\State $S \leftarrow P(1:k)$ \# Choose the first $k$ elements in $P$
		\State return $S$
	\end{algorithmic}
\end{algorithm}

\begin{algorithm}[h]
	\caption{\textsf{SparseAnchorSelection}: Select anchor points using sparse randomized interpolative decomposition approach (RRQR)}
	\label{alg:sparse_ID_Sampling}
	\begin{algorithmic}[1]
		\REQUIRE $X$ - Dataset, matrix of size $n \times d$.\\
		$k$ - Number points to choose (rank), \\
		$l \ge k$ - number of random projections, typically slightly larger than $k$\\
		$h>0$ - Width of the Gaussian kernel\\
		$r$ - Number of non-zero entries in each column of the random projection matrix
		\ENSURE $S$ - A set of anchor points from $X$
		\State $Y \leftarrow zeros(n,l)$
		\For{$i$ in range$(l)$}
		\State Define a set $R$ of $r$ values chosen randomly from $\{1,\ldots,n\}$
		\State Define a vector $v$ of length $r$, where each entry is taken from a Rademacher distribution
		\State $X_R \leftarrow X(R,:)$
		\State $Y(:, i) \leftarrow$ FIG($X$,$X_R$, $h$, $v$) \# Apply FIG transform to $v$
		\EndFor
		\State $Y^TP=QR$ \# Apply Strong RRQR to $Y$, $P$ is a permutation matrix, but can viewed as a vector
		\State $S \leftarrow P(1:k)$ \# Choose the first $k$ elements in $P$
		\State return $S$
	\end{algorithmic}
\end{algorithm}
Algorithm \ref{alg:ID_Sampling} computes a randomized ID to the full kernel matrix. Applying a random matrix $\Omega$ to the full kernel, 
by using the FIG transform (line 7) and the pivoted QR on the matrix $Y^T$ returns the columns that gives good low rank approximation, 
according to the ID, i.e. the matrix $B$ from Section \ref{sec:id_cur}. The interpolation matrix, $T$ is not needed and therefore not computed.

\begin{algorithm}
	\caption{\textsf{BuildCholesky}: Computes the Cholesky decomposition according to Eq. \ref{eq:chol_factor}}
	\label{alg:BuildChol}
	\begin{algorithmic}[1]
		\REQUIRE $X$ - Dataset, matrix of size $n \times d$,\\
		$S$ - Vector of length $k$, \\
		$\beta>0$ - Ridge parameter,\\
		$h>0$ - Width of the Gaussian kernel
		\ENSURE $L$ - Cholesky factor
		\State $X_S \leftarrow X(S,:)$, \# $X_S$ is a subset of $X$ containing the anchor points
		\State $U^{-1} \leftarrow FIG(X_S, X_S, h, I_{k\times k})$
		\State Using SVD or EVD build $U^{\frac{1}{2}}$
		\State $Y \leftarrow FIG(X, X_S, h, U^{\frac{1}{2}})$
		\State $Y \leftarrow \beta I + U^{\frac{1}{2}}FIG(X_S, X, h, Y)$		\State $ Y = L^TL$ \# Cholesky decomposition
		\State return $L, U^{\frac{1}{2}}$
	\end{algorithmic}
\end{algorithm}

\begin{algorithm}[ht]
	\caption{\textsf{ApplyPreconditioner}: Applies the preconditioner according to Eq. \ref{eq:woodbury}}
	\label{alg:ApplyPrecond}
	\begin{algorithmic}[1]
		\REQUIRE $X$ - Dataset, matrix of size $n \times d$.\\
		$S$ - Anchor points, matrix of size $k \times d$ \\
		$\beta>0$ - Ridge parameter\\
		$h>0$ - Width of the Gaussian kernel\\
		$L$ - Cholesky factor\\
		$x$ - Input vector of length $n$.\\
		$U^{\frac{1}{2}}$ - Output from Algorithm \ref{alg:BuildChol}.
		\ENSURE $r$ - The application of the preconditioner $(\tilde{K}+\beta I)^{-1}$ to $x$.
		\State $X_S \leftarrow X(S,:)$, \# $X_S$ is a subset of $X$ containing the anchor points
		\State $y \leftarrow U^{\frac{1}{2}}$ FIG($X_S$, $X$, $h$, $x$)
		\State Solve $L^TLz'=y$ \# back substitution 
		\State $z \leftarrow U^{\frac{1}{2}} z'$
		\State $r \leftarrow \beta^{-1}x-\beta^{-1}$ FIG($X$, $X_S$, $h$, $z$)		\State return $r$
	\end{algorithmic}
\end{algorithm}

\begin{algorithm}[h]
	\caption{\textsf{SolveKRR}: Solves the Gaussian kernel ridge regression}
	\label{alg:SolveKRR}
	\begin{algorithmic}[1]
		\REQUIRE $X$ - Dataset, matrix of size $n \times d$.\\
		$\beta>0$ - Ridge parameter\\
		$h>0$ - Width of the Gaussian kernel\\
		$k$ - Number of points to sample.\\
		$l \ge k$ - Number of random projections to use.\\
		$b$ - vector of length $n$.
		\ENSURE $x$ - The solution to the equation $(K+\beta I)x=b$.
		\State Select anchor points $S$ from $X$ using Algorithm \ref{alg:ID_Sampling} or Algorithm \ref{alg:sparse_ID_Sampling}.
		\State Build the Cholesky factor $L$ using Algorithm \ref{alg:BuildChol}.
		\State Solve $(K+\beta I)x=b$, using CG, where $Kx$ can be computed by FIG($X$,$X$,$h$, x) and the preconditioner can be applied using Algorithm \ref{alg:ApplyPrecond}
		\State return $x$
	\end{algorithmic}
\end{algorithm}
Algorithm \ref{alg:sparse_ID_Sampling} is similar to Algorithm \ref{alg:ID_Sampling} except that a sparse random projection matrix is used, and the distribution is not normal. The algorithm uses a random projection matrix that has only $r$ non zero entries in each column. Therefore, the projection can be obtained by applying a thin FIG transform of size $n \times r$ to a vector of length $r$ which reduces the computational complexity. This is done $l$ times. After this projection step, the next steps of the algorithm are the same as in Algorithm \ref{alg:ID_Sampling}, i.e. computing pivoted QR and returning the anchor points. 
\begin{rem}
	\label{rem:sparse_selection}
	The computational complexity of Algorithm \ref{alg:sparse_ID_Sampling} is $\mathcal{O}((n+r)l+nl^2)$, typically $ r \ll l$.
\end{rem}
\begin{rem}
	Algorithm \ref{alg:ID_Sampling} uses sampling technique based on interpolative decomposition. When using pivoted QR, the computational complexity is $\mathcal{O}(nl^2)$. In practice, the performance is very similar to RRQR.
\end{rem}
\begin{rem}
	The computational complexity of Algorithm \ref{alg:SolveKRR}, when using pivoted QR in Algorithm \ref{alg:ID_Sampling} is $\mathcal{O}(nl^2+nk+k^3+(n+k^2)t)$, where $t$ is the number of iterations used in the conjugate gradient.
\end{rem}
\begin{rem}
	The computational complexity of Algorithm \ref{alg:BuildChol} is $\mathcal{O}(nk+k^3)$
\end{rem}
\begin{rem}
	Applying the preconditioner (Algorithm \ref{alg:ApplyPrecond}) in each iteration takes $\mathcal{O}(n+k^2)$ operations.
\end{rem}
\begin{rem}
	Anchor points selection can be done differently, for example by random sampling or FPS. In this case, the theoretical bounds will not hold, but may still produce good results in practice and reduce the computational cost of Algorithm \ref{alg:ID_Sampling}. For example, from $\mathcal{O}(nl^2)$ using pivoted QR, to $\mathcal{O}(k)$ using random sampling.
\end{rem}

\section{Theoretical Analysis}
\label{sec:theory}
This section presents the theoretical analysis of the algorithm. Specifically, Theorem \ref{th:maintheorem} gives an upper bound on the condition number of the preconditioned system and Corollary \ref{cor:tradeoff}  specifies the trade-off between the number of anchor points selected and the upper bound of the condition number.

The following Lemma gives an error bound to the \Nystrom approximation computed by interpolative decomposition:
\begin{lemma}
	\label{lem:era}
	Let $K \in \mathbb{R}^{n\times n}$ be a symmetric matrix, where $C$ are subset of columns chosen using an interpolative decomposition and let $\tilde{K}=CUC^T$ be its Nystr\"{o}m approximation, then $\vert \vert K-\tilde{K} \vert \vert \le 2\sigma_{k+1}(K)\cdot p(k,n)$ where p(k,n) is a function bounded by a low degree polynomial in $k$ and $n$.
\end{lemma}
\begin{proof}
	By Lemma \ref{lemma:curid}  
	\begin{equation}
	\vert \vert K-\tilde{K} \vert \vert \le \vert \vert E \vert \vert + \vert \vert \tilde{E} \vert \vert
	\end{equation}
	Note however that since K is symmetric, $R=C^{T}$ meaning
	\begin{equation}
	\tilde{K}=CUR=CUC^{T} 
	\end{equation} 
	hence $ \Vert E \Vert = \Vert \tilde{E} \Vert$.
	Therefore it follows immediately that
	\begin{equation}
	\label{eq:lemmaeq1}
	\Vert K - \tilde{K} \Vert \le 2\vert \vert E \vert \vert
	\end{equation} 
	By the definition of RRQR, it is known that the decomposition must satisfy
	\begin{equation}
	\label{eq:lemmaeq2}
	\Vert E \Vert \le \sigma_{k+1}(K)\cdot p(k,n).
	\end{equation}
	Therefore, combining Eq. \ref{eq:lemmaeq1} with Eq. \ref{eq:lemmaeq2} yields
	\begin{equation}
	\Vert K - \tilde{K} \Vert \le 2\sigma_{k+1}(K)\cdot p(k,n)
	\end{equation}	\end{proof}
\begin{lemma}
	\label{lemma:schur}
	Let $K \in \mathbb{R}^{n\times n}$ be a positive semidefinite matrix, and let $\tilde{K}=CUC^T$ its Nystr\"{o}m approximation, then $K-\tilde{K} \ge 0$.
\end{lemma}
\begin{proof}
	The lemma follows directly from the Schur's complement of $K$.
\end{proof}
\begin{lemma}[\cite{bernstein2009matrix} pp. 673]
	\label{lemma:book}
	For any two matrices $A, B \in \mathbb{C}^{m \times n}$ the following holds:
	\begin{equation}
	\label{eq:book_lemma}
	\vert \sigma_i(A)-\sigma_i(B) \vert \le \Vert A-B \Vert
	\end{equation}
\end{lemma}
\begin{definition}\cite{bermanis2013multiscale}
	\label{def:numerical_rank}
	The numerical rank of the Gaussian kernel $G_h^X$ up to precision $\delta \ge 0$ is
	\begin{equation*}
	\rho_\delta (G_h^X) \triangleq \# \left[j: \frac{\sigma_j(G_h^X)}{\sigma_1(G_h^X)} \ge \delta \right]
	\end{equation*}
\end{definition}
\begin{theorem}\cite{bermanis2013multiscale,bermanis2014cover}
	\label{th:gaussiankernelbound}
	Let $X=\{x_i\}_{i=1}^n \in \mathbb{R}^d$ be a set bounded by a box $B=I_1 \times I_2, \cdots I_d$, where $I_1, I_2, \cdots, I_d$
	are intervals in $\mathbb{R}$. Let $q_i$ be the length of the $i$th interval, i.e. $q_i=\vert I_i \vert$ and let $G_h^X$ be the associated kernel matrix, i.e. $(G_h^X)_{i,j}=g_h(x_i,x_j)=e^{\frac{\Vert x_i-x_j\Vert^2}{h^2}}$, then
	\begin{equation}
	\rho_\delta (G_h^X) \le \prod_{i=1}^d \left( \floor{\gamma q_i}+1\right),
	\end{equation}
	where $\gamma=\frac{2}{\pi}\sqrt{h^{-2}\ln{\delta^{-1}}}$.
\end{theorem}
The following theorem gives an upper bound to the condition number:
\begin{theorem}
	\label{th:maintheorem}
	Let $\tilde{K} \in \mathbb{R}^{n \times n}$ be a rank $k$ Nystr\"{o}m approximation for a positive semidefinite matrix $K \in \mathbb{R}^{n \times n}$, such that $\Vert K-\tilde{K} \Vert \le M\sigma_{k+1}(K)$, for a positive constant $M$ (that may depend on $n$ and $k$) and a ridge parameter $\beta>0$, then 
	\begin{equation}
	\textnormal{cond}\left[ (\tilde{K}+\beta I)^{-1}(K+\beta I)\right] \le 1+\frac{M\lambda_{k+1}(K)}{\beta}
	\end{equation}
\end{theorem}
\begin{proof}
	Combining Lemma \ref{lemma:schur} and Lemma \ref{lemma:book} with the fact that $K$ and $\tilde{K}$ are positive semidefinite, gives
	\begin{equation}
	\label{eq:eq1}
	0 \le K-\tilde{K} \le \Vert K-\tilde{K}\Vert I \le M\sigma_{k+1}(K)=M\lambda_{k+1}(K)
	\end{equation}
	Modifying Eq. \ref{eq:eq1} and adding $\beta I$ gives
	\begin{equation}
	\tilde{K}+\beta I \le K +\beta I\le \tilde{K}+\beta I+M\lambda_{k+1}(K)I
	\end{equation}
	Applying $(\tilde{K}+\beta I)^{-\frac{1}{2}}$ to both sides of the equation and using the fact that $(\tilde{K}+\beta I)^{-\frac{1}{2}}$ is symmetric positive semidefinite gives
	\begin{equation}
	\begin{split}
	I \le &(\tilde{K}+\beta I)^{-\frac{1}{2}}(K +\beta I)(\tilde{K}+\beta I)^{-\frac{1}{2}}
	\\ & \le
	I+M\lambda_{k+1}(K)(\tilde{K}+\beta I)^{-1}
	\end{split}
	\end{equation}
	Clearly, $(\tilde{K}+\beta I)^{-1} \le \beta^{-1}I$, which yields
	\begin{equation}
	\begin{split}
	I &\le (\tilde{K}+\beta I)^{-\frac{1}{2}}(K +\beta I)(\tilde{K}+\beta I)^{-\frac{1}{2}}\\
	&\le (1+\frac{M\lambda_{k+1}(K)}{\beta})I.
	\end{split}
	\end{equation}
	and finally,
	\begin{equation}
	\begin{split}
	&\text{cond}\left[ (\tilde{K}+\beta I)^{-1}(K+\beta I)\right]=\\
	&\text{cond}\left[ (\tilde{K}+\beta I)^{-\frac{1}{2}}(K+\beta I)(\tilde{K}+\beta I)^{-\frac{1}{2}}\right] \le \\
	& 1+\frac{M\lambda_{k+1}(K)}{\beta}
	\end{split}
	\end{equation}
	which completes the proof.
\end{proof}
\begin{corollary}
	\label{cor:tradeoff}
	
	Let $K$ be a Gaussian kernel matrix over a dataset $X \in \mathbb{R}^d$, i.e. $K=G_h^X$. Let $\tilde{K}$ be a rank $k$ Nystr\"{o}m approximation, such that $\Vert K-\tilde{K}\Vert \le M(n,k)\sigma_{k+1}(K)$ and let $\beta$ be a ridge parameter. Then, for a maximal condition number, $\xi$ 
	\begin{equation}\label{eq:condbound}
	k \ge \prod_{i=1}^d \left( \floor{\frac{2q_i}{\pi}\sqrt{h^{-2}\ln\frac{\bar{M}n}{\beta(\xi-1)}}}+1\right)
	\end{equation}
	where $\{q_i\}_{i=1}^d$ are the lengths of the intervals of the bounding box of the dataset $X$.
\end{corollary}
\begin{proof}
	From Theorem \ref{th:maintheorem}, the condition number depends on $\lambda_{k+1}(K)$, i.e.
	\begin{equation}
	\text{cond}\left[ (\tilde{K}+\beta I)^{-1}(K+\beta I)\right]\le 1+\frac{\bar{M}\lambda_{k+1}(K)}{\beta} \le \xi
	\end{equation}
	where $\bar{M} \triangleq \sup{M(k)}$.
	Therefore, $\lambda_{k+1}(K)$ must satisfy
	$\lambda_{k+1}(K) \le \frac{\beta(\xi-1)}{\bar{M}}$. By defining $\delta=\frac{\beta(\xi-1)}{\bar{M}\Vert K \Vert}$ and using Theorem \ref{th:gaussiankernelbound}, 
	\begin{equation}
	\label{eq:condbound_gamma}
	\gamma=\frac{2}{\pi}\sqrt{h^{-2}\ln(\delta^{-1})} \le \frac{2}{\pi}\sqrt{h^{-2}\ln\frac{\bar{M}n}{\beta(\xi-1)}}
	\end{equation}
	and therefore,
	\begin{equation*}
	k \ge \prod_{i=1}^d \left( \floor{\frac{2q_i}{\pi}\sqrt{h^{-2}\ln\frac{\bar{M}n}{\beta(\xi-1)}}}+1\right)
	\end{equation*}
	where $\{q_i\}_{i=1}^d$ are the lengths of the intervals of the bounding box of the dataset $X$.
\end{proof}
\begin{rem}
	Corollary \ref{cor:tradeoff} enables to determine the required rank Nystr\"{o}m approximation and illustrates the tradeoff between 
	the condition number and the low rank, which has implications on the memory consumption and computational load of the algorithm. Smaller condition number yields less iterations on one hand, but on the other hand requires processing a larger matrix $\tilde{K}$. This corollary applies strictly for Gaussian kernels.
\end{rem}
\begin{rem}
	As a simple example, for the deterministic interpolative decomposition, $M=\sqrt{4k(n-k)+1}$, which yields $\bar{M}=\sqrt{n^2+1} \approx n$
\end{rem}
\begin{rem}
	Eq. \ref{eq:condbound} grows very slowly in $n$, which means that $k$ remains small even when $n$ grows fast. For example, suppose $\bar{M}=n$, $h=1$, $\xi=2$ and $\beta=1$, then for $n=10^6$, $\gamma=3.34$ and for $n=10^8$, $\gamma=3.86$ which does not affect the value of $k$ according to Corollary \ref{cor:tradeoff}. On the down side, it grows fast in $d$. 
	For example, for the \texttt{Buzz} dataset (see below), $q_i=1, h=1, \beta=0.1, n=400,000, d=77$ choosing $\xi=4$ leads to $k=4^{77} \approx 10^{46}$, which is very practical for real applications. 
	Though, it is interesting for its theoretical value, that links between the condition number and the size of the preconditioner.
\end{rem}
\begin{corollary}
	For a Nystr\"{o}m preconditioner built by choosing $\tilde{K}=CUC^T$, where $C$ are columns of $K$ chosen by the interpolative decomposition, then
	\begin{equation}
	\begin{split}
	\textnormal{cond}&\left[ (\tilde{K}+\beta I)^{-1}(K+\beta I)\right] \\
	&\le 1+\frac{2\lambda_{k+1}(K)\sqrt{4k(n-k)+1}}{\beta}\\
	&\le 1+\frac{2\lambda_{k+1}(K)(n+1)}{\beta}
	\end{split}
	\end{equation}
\end{corollary}
\begin{proof}
	The proof follows immediately by combining Lemma \ref{lem:era} with Theorem \ref{th:maintheorem}.
\end{proof}
\begin{rem}
	A similar bound can be developed immediately for other matrix decompositions. For example, if $\tilde{K}$ is the best rank $k$ approximation of $K$, computed by the singular value decomposition, then direct substitution gives
	\begin{equation*}
		\textnormal{cond}\left[ (\tilde{K}+\beta I)^{-1}(K+\beta I)\right] = \frac{\lambda_{k+1} + \beta}{\beta}=1+\frac{\lambda_{k+1}}{\beta}.
		\end{equation*}
\end{rem}

\subsection{Convergence Rate Analysis}
\label{sec:convergence_rate}
This subsection discusses the rate of convergence and the decrease in the number of iterations when using the preconditioner with respect to the non-preconditioned KRR. The analysis is done by using the theory from Section \ref{sec:theory} with facts on the convergence rate of the conjugate gradient algorithm \cite{golub2012matrix}. Formally, what is the required number of iterations, using conjugate gradient, for solving
\begin{equation}
\label{eq:krr_noprecond}
	(K+\beta I)x=b
\end{equation}
with respect to
\begin{equation}
	\label{eq:krr_precond}
	(\tilde{K}+\beta I)^{-1}(K+\beta I)x=(\tilde{K}+\beta I)^{-1}b
\end{equation}
Let $x_i$ and $x_i^{\text{precond}}$ be the solutions obtained after $i$ CG iterations for solving Eqs. \ref{eq:krr_noprecond} and \ref{eq:krr_precond}, respectively.
Similarly, $e_i$ and $e_i^{\text{precond}}$ are the corresponding errors with respect to the exact solution $x_*$, i.e. 
		$e_i = x_* - x_i$. Suppose $\kappa$ is the condition number of Eq. \ref{eq:krr_noprecond}, then the error in the $i$th iteration is given by 
\begin{equation}
	\label{eq:cgrate}
	\Vert e_i \Vert_{K+\beta I}=2\Vert e_0 \Vert_{K+\beta I} \left( \frac{\sqrt{\kappa}-1}{\sqrt{\kappa}+1} \right)^i \approx \Vert e_0 \Vert_{K+\beta I} \left( 1-\frac{2}{\sqrt{\kappa}} \right)^i .
\end{equation}
Where the last transition holds for large values of $\kappa$, which is usually the case for non-preconditioned systems. Next, suppose we wish to achieve a relative error $\delta$, such that $\Vert e_i \Vert_{K+\beta I}/\Vert e_0 \Vert_{K+\beta I} \le \delta$. From Eq. \ref{eq:cgrate} and using the fact that $\kappa=\frac{\lambda_1+\beta}{\lambda_n+_\beta}\approx \frac{\lambda_1}{\beta}$ is large, the number of iterations should be
\begin{equation}
\label{eq:noprecond_rate}
	i = \frac{\ln{\delta}}{\ln\left({1-\frac{2}{\sqrt{\kappa}}}\right)} \approx -\frac{\sqrt{\lambda_1}\ln{\delta}}{2\sqrt{\beta}}.
\end{equation}
Eq. \ref{eq:noprecond_rate} gives a good estimation for the number of iterations for the non-preconditioned system, depending on the largest eigenvalue of $K$. Similarly, for the preconditioned system, using Eq. \ref{eq:cgrate} and substituting the condition number $\kappa=1+\frac{M\lambda_{k+1}}{\beta}$ from Theorem \ref{th:maintheorem} gives
\begin{equation}
\label{eq:precond_rate1}
i^{\text{precond}}\ln\left( \frac{\sqrt{1+\frac{M\lambda_{k+1}}{\beta}}-1}{\sqrt{1+\frac{M\lambda_{k+1}}{\beta}}+1} \right)=\ln{\delta}
\end{equation}
Doing some algebraic manipulations, gives 
\begin{equation*}
	i^{\text{precond}}\ln\left( \frac{\sqrt{1+\frac{M\lambda_{k+1}}{\beta}}-1}{\sqrt{1+\frac{M\lambda_{k+1}}{\beta}}+1} \right)=i^{\text{precond}}\ln\left(\frac{\sqrt{\beta+M\lambda_{k+1}} - \sqrt{\beta}}{\sqrt{\beta+M\lambda_{k+1}} + \sqrt{\beta}} \right) 
\end{equation*}
and finally,
\begin{equation}
	\label{eq:precond_rate2}
	i^{\text{precond}} = \frac{\ln{\delta}}{\ln\left(\frac{\sqrt{\beta+M\lambda_{k+1}} - \sqrt{\beta}}{\sqrt{\beta+M\lambda_{k+1}} + \sqrt{\beta}} \right) }
\end{equation}
Combining Eq. \ref{eq:noprecond_rate} with Eq. \ref{eq:precond_rate2} gives the factor of the improvement in the number of iterations


\begin{equation}
	\label{eq:ratio}
	i/i^{\text{precond}}=-\frac{\sqrt{\lambda_1}\ln\left(\frac{\sqrt{\beta+M\lambda_{k+1}} - \sqrt{\beta}}{\sqrt{\beta+M\lambda_{k+1}} + \sqrt{\beta}} \right) }{2\sqrt{\beta}} \approx -\frac{\sqrt{\lambda_1}}{2\sqrt{\beta}}\ln \left(\frac{M\lambda_{k+1}}{4\beta}\right),
\end{equation}
where the last transition is approximated by assuming that $M\lambda_{k+1}$ is small.
The following theorem is taken from \cite{ma2017diving}:
\begin{theorem}\cite{ma2017diving}
\label{th:diving}
If $k(x,y)$ is an infinitely differentiable kernel, the rate of eigenvalue decay is super-polynomial,
i.e $\lambda_j = \mathcal{O}(j^{-P})~~\forall P\in \mathbb{N}$. 
\end{theorem}
\begin{rem}
By using Theorem \ref{th:diving} with Eq. \ref{eq:ratio} it is possible to understand the asymptotic improvement in the number of iterations. For an infinitely differentiable kernel,
\begin{equation*}
	i/i^{\text{precond}}=\mathcal{O}(P\ln k).
\end{equation*}
\end{rem}

\section{Numerical Results}
The algorithm was tested in several scenarios, involving small and large datasets with low and high dimensions. The comparison focuses on the number of iterations and not on absolute time for two reasons: First, the goal is to reduce the number of iterations (preconditioning). Second, the FIG-transform used is not optimized and does not use state-of-the-art software packages for nearest neighbor search nor does it use MKL so absolute times are not very meaningful. Two types of experiements were conducted:
\begin{itemize}
	\item Low dimensional datasets, where the FIG-transform was very efficient, running on a \textit{single core CPU}. The performance of the preconditioner was compared to a random Fourier features (``RFF") based preconditioner, as described in \cite{avron_krr}. 
	The anchor selection was done using Algorithm \ref{alg:ID_Sampling} for the Helicopter dataset and using Algorithm \ref{alg:sparse_ID_Sampling} for the electric field dataset.
	\item High dimensional datasets, running in Amazon AWS cloud on a single EC2 c5.24xlarge machine with multiple cores and anchor selection using Algorithm \ref{alg:sparse_ID_Sampling}.
\end{itemize}
The error was computed in $l_2$ norm, and in the inner-product norm $\Vert \cdot \Vert_{K+\beta I}$.
\subsection{Performance of the Preconditioner}
In this section, the singular values of the kernel were computed with and without the preconditioner. Due to computational limits, which involved computing exact SVD on large matrices, it was done on a subset of random samples from the data and applied to a matrix of size $10,000 \times 10,000$. Clearly, the condition number of the full matrix is much larger, but this still gives an intuition on the performance of the preconditioner.
\begin{figure}[H]
	\centering
	\includegraphics[scale=0.7]{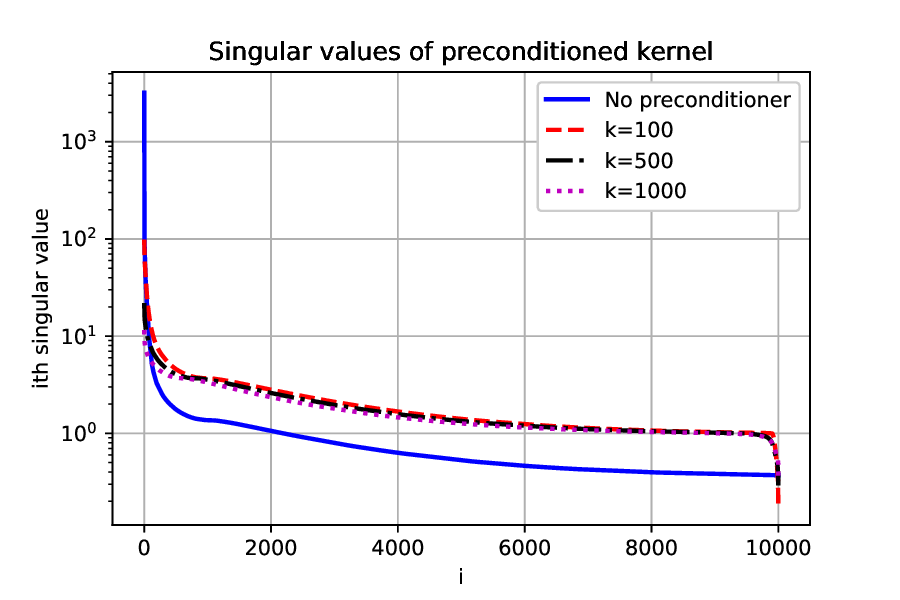}
	\caption{Singular Values with and without preconditioner for \texttt{YearSong} dataset with $h=9$ and $\beta=0.37$}
	\label{fig:sv_yearMSD}.
\end{figure}

Figure \ref{fig:sv_yearMSD} shows the decay of the kernel for the \texttt{YearSong} dataset. The anchor points were selected using Algorithm \ref{alg:sparse_ID_Sampling}, where $l=k+5$. The condition number without preconditioner was $8,545$ and the condition number for $k=100, k=500$ and $k=1000$ was $548,~92$ and $40$, respectively.

The same analysis was applied to the \texttt{CovType} dataset. The decay of the singular values appear in Figure \ref{fig:sv_CovType}.
\begin{figure}[H]
	\centering
	\includegraphics[scale=0.7]{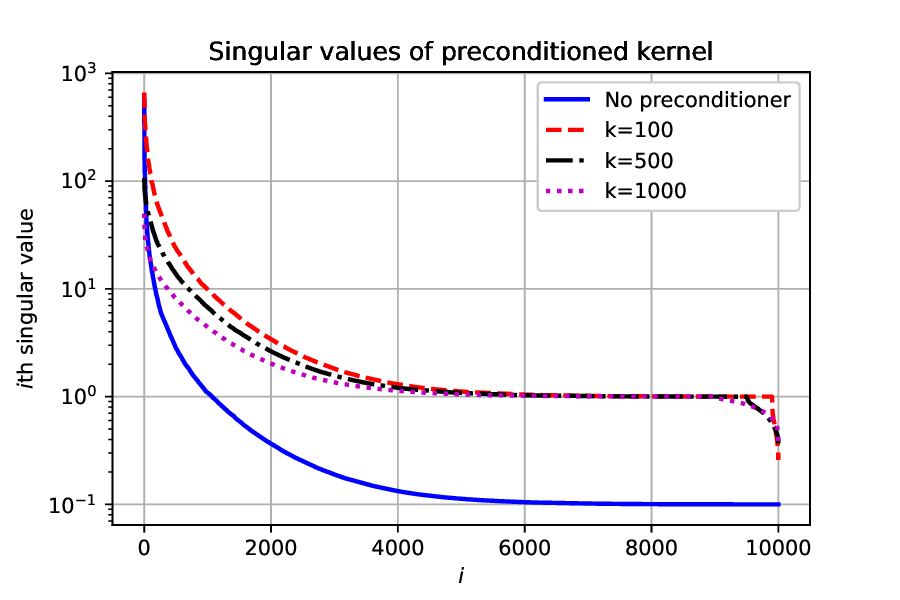}
	\caption{Singular Values with and without preconditioner for \texttt{CovType} dataset with $h=3$ and $\beta=0.1$}
	\label{fig:sv_CovType}.
\end{figure}

Figure \ref{fig:sv_CovType} shows the decay of the kernel for the \texttt{CovType} dataset. The anchor points were selected using Algorithm \ref{alg:sparse_ID_Sampling}, where $l=k+5$. The condition number without preconditioner was $5,677$ and the condition number for $k=100, k=500$ and $k=1000$ was $2556,~289$ and $128$, respectively.

\begin{figure}[H]
	\centering
	\includegraphics[scale=0.7]{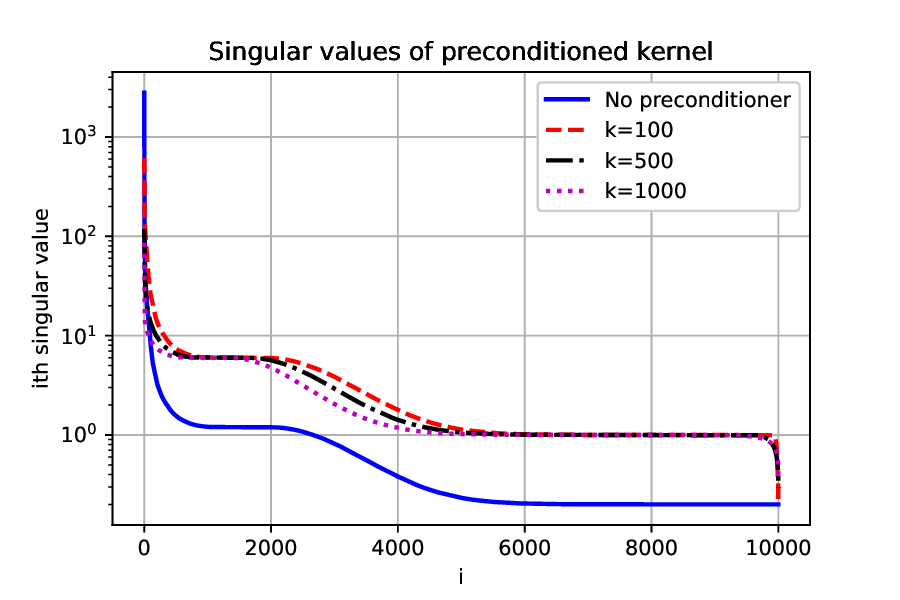}
	\caption{Singular Values with and without preconditioner for \texttt{Buzz} dataset with $h=1$ and $\beta=0.2$}
	\label{fig:sv_Buzz}.
\end{figure}

Figure \ref{fig:sv_Buzz} shows the decay of the kernel for the \texttt{Buzz} dataset. The anchor points were selected using Algorithm \ref{alg:sparse_ID_Sampling}, where $l=k+5$. The condition number without preconditioner was $14,000$ and the condition number for $k=100, k=500$ and $k=1000$ was $2832,~403$ and $242$, respectively.

\subsection{Low Dimensional Datasets}
In this section, empirical evaluation of the algorithm is presented and compared against the naive implementation (no preconditioning) and 
RFF preconditioning (\cite{avron_krr}). Comparing with \cite{avron_krr} is relatively straightforward, since both algorithms are similar as they both solve the KRR problem accurately by using a preconditioner.
Part of the RFF algorithm can also be implemented using the FIG transform, while the algorithm suggested here, can be implemented 
by direct matrix multiplication. The Helicopter dataset described in \cite{helicopter} and the electric field dataset was 
generated by a simulator who generated a simple 3D field of $n=100,000$ points. 
The most costly computation in the preprocessing of the presented flow is the QR decomposition in Algorithm \ref{alg:ID_Sampling} and Algorithm \ref{alg:sparse_ID_Sampling} which is $\mathcal{O}(nl^2)$ and the cost of each iteration is of $\mathcal{O}(n+k^2)$. 
In the RFF method (\cite{avron_krr}) the preprocessing takes $\mathcal{O}(ns^2)$ operations and each iteration takes $\mathcal{O}(ns)$ operations, where $s$ is number of sine and cosine components in the RFF (i.e. the size of the preconditioner). 
Throughout the experiments, we made sure that the computational cost of the method suggested in this paper, is lower than the computational budget that was given to the RFF method. For simplicity, we used the same number of iterations for both methods.
Figure \ref{fig:hel_100_main} shows the comparison between the distance $\Vert \alpha^*-\alpha^{(i)}\Vert_{K+\beta I}$, where $\alpha^*$ is the exact solution, and $\alpha^{(i)}$ is the solution at iteration $i$. 
\begin{figure}[htp]
	\centering
	\begin{subfigure}[b]{0.45\linewidth}
		\includegraphics[width=\linewidth,height=5cm]{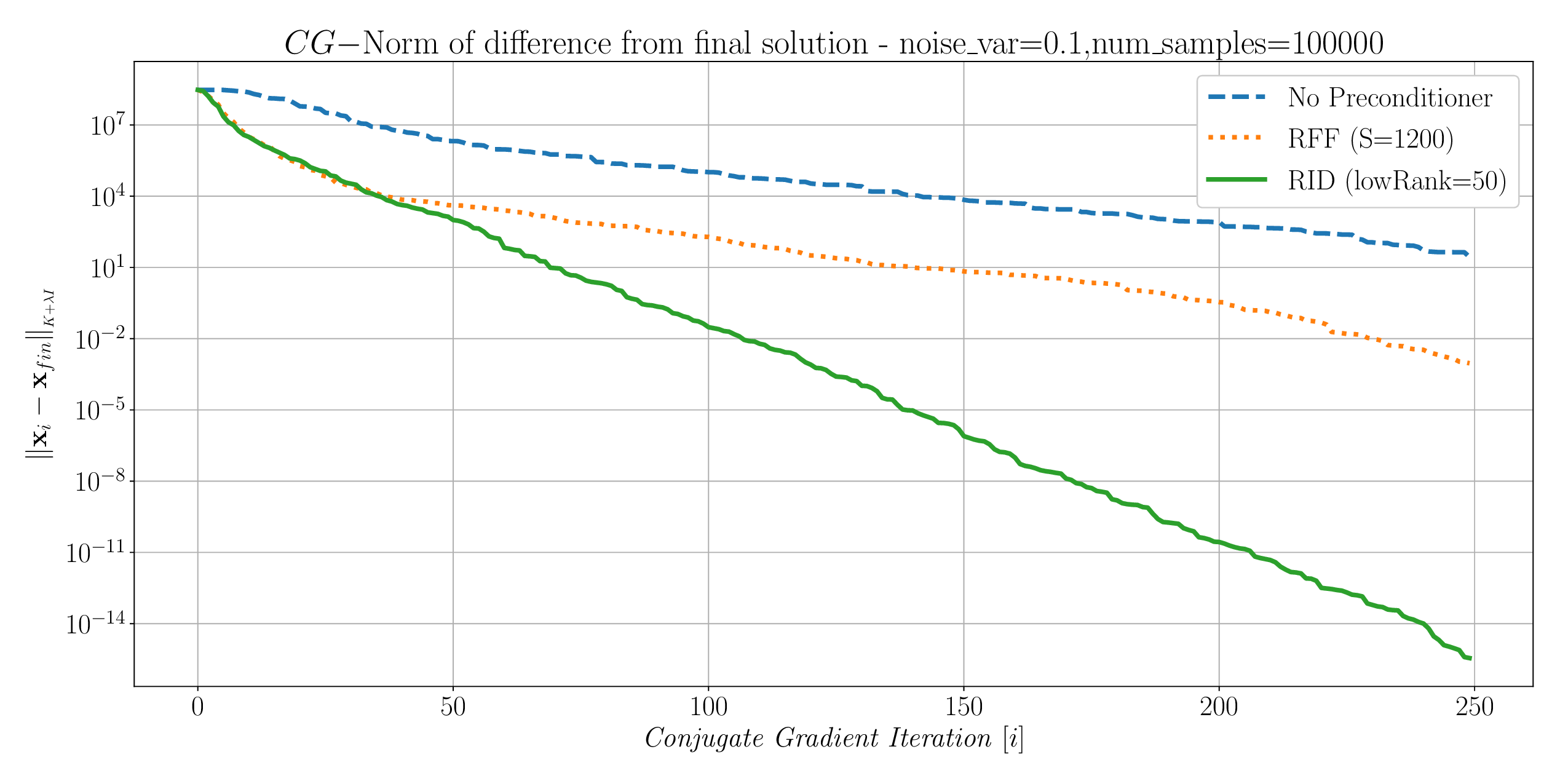}
		\caption{$\beta=0.1$}
		\label{fig:hel_50}
	\end{subfigure}
	\begin{subfigure}[b]{0.45\linewidth}
		\includegraphics[width=\linewidth,height=5cm]{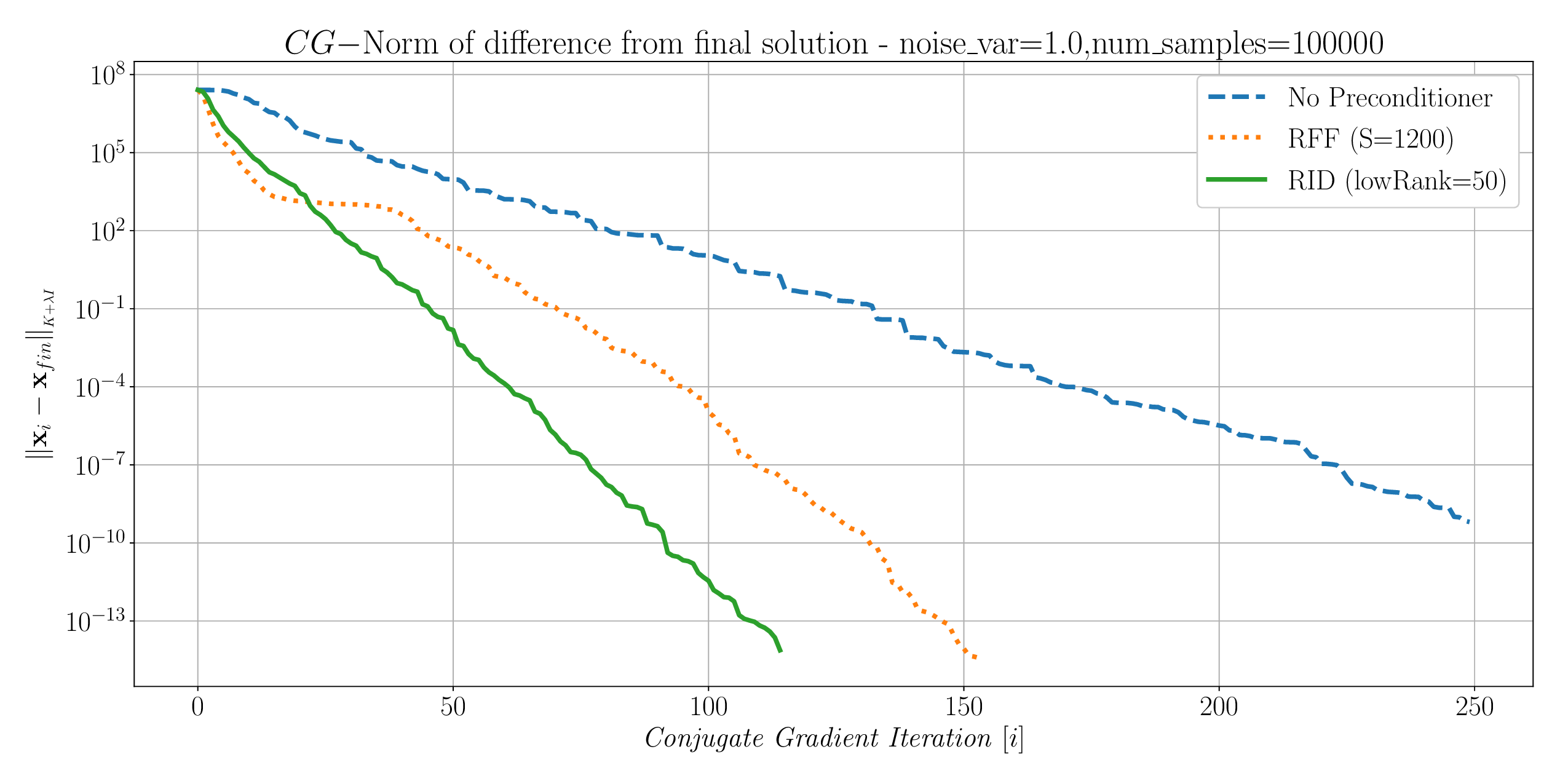}	
		\caption{$\beta=1$}
		\label{fig:hel_100}
	\end{subfigure}
	\caption{Stanford Helicopter: Comparison between different noise levels, $k=50$}
	\label{fig:hel_100_main}
\end{figure}
\begin{rem}
	Each iteration of the proposed algorithm is $\mathcal{O}(n+k^2)$, while every iteration of RFF is $\mathcal{O}(ns)$. For $n \gg k^2$, each RFF iteration is asymptotically more costly in terms of run-time complexity by a factor of $s$.
\end{rem}
\begin{rem}
	When the spectral decay is fast, a small $k$ will be sufficient for satisfying results.
\end{rem}
\begin{figure}[ht]
	\centering
	\begin{subfigure}[b]{0.45\linewidth}
		\includegraphics[width=\linewidth,height=7cm]{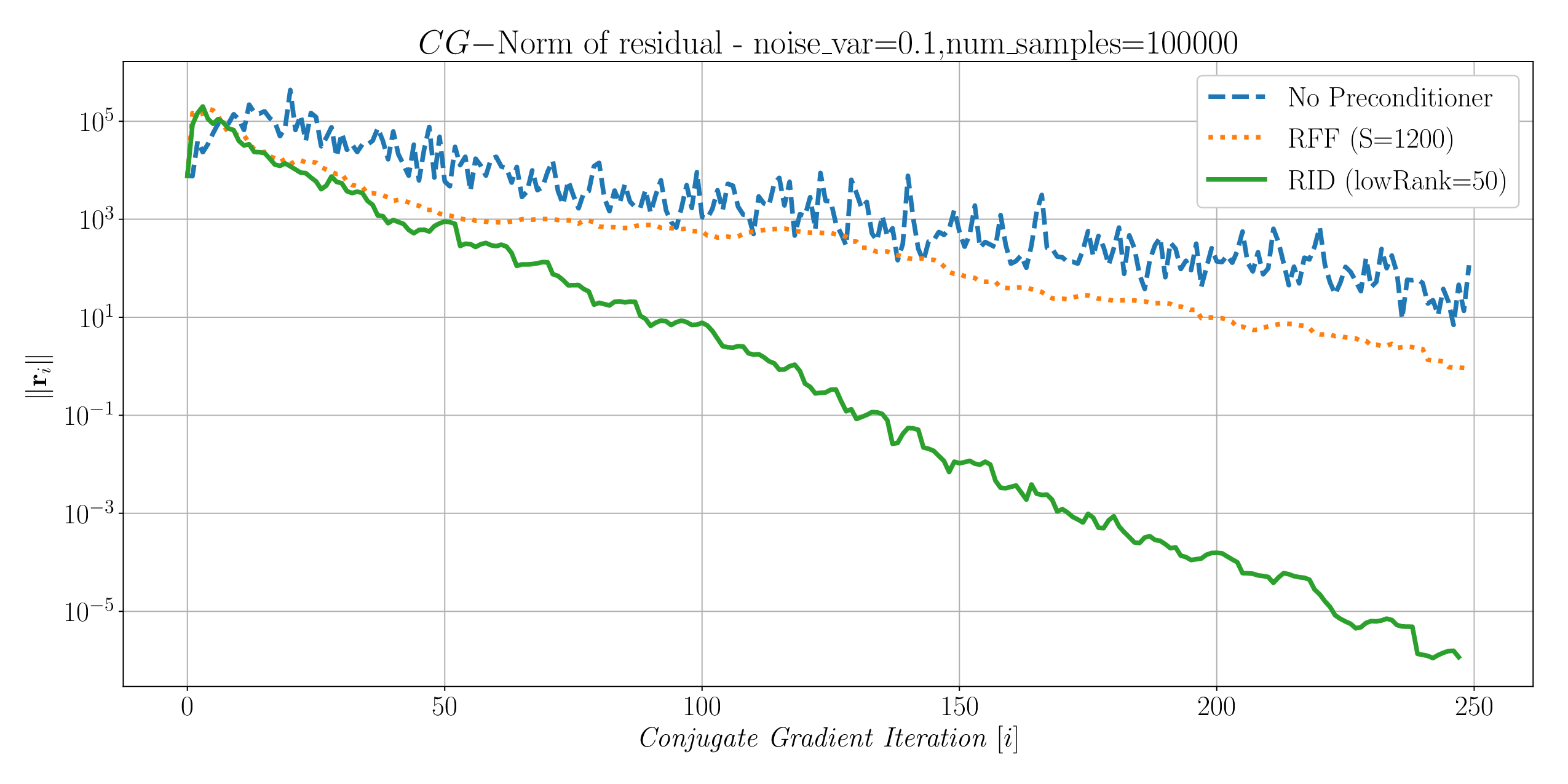}
		\caption{Stanford Helicopter Dataset $k=50$, $\beta=0.1$}
		\label{fig:hel_50_r}
	\end{subfigure}
	\begin{subfigure}[b]{0.45\linewidth}
		\includegraphics[width=\linewidth,height=7cm]{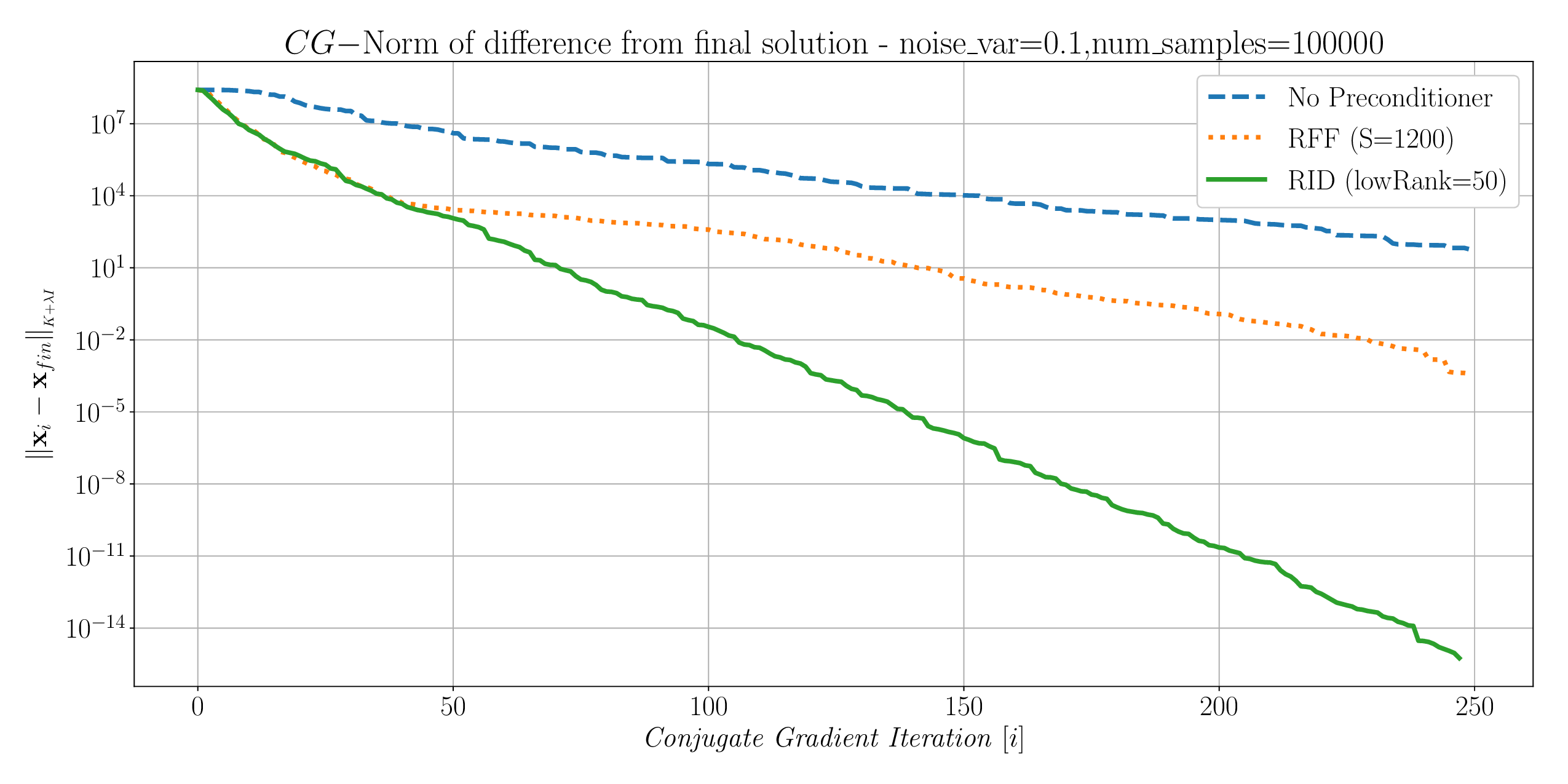}
		\caption{Stanford Helicopter Dataset $k=50$, $\beta=0.1$}
		\label{fig:hel_50_sol}
	\end{subfigure}
	\caption{For 100K datapoints from accelerating 2D (X,Y) predicting next timestep for Y}
\end{figure}
\subsubsection{Synthetic Electromagnetic Field Simulator}
In addition to the Helicopter dataset, a three dimensional synthetic dataset generated by an electric field simulator was used to test the performance of the algorithm. The simulation generated a simple electric field of $5$ charges in the space. The field was sampled in $100,000$ points and the goal was to estimate the electric potential at each of those points. Two experiments were conducted: For the first one, whose results are shown in Fig. \ref{fig:EM_RFF30_100k}, the random Fourier matrix was chosen with $s=180$. For the second experiment, whose results are shown in Fig. \ref{fig:EM_RFF150_100k}, more random features was used, so that $s=900$. It is interesting to note, that for the case of $s=180$, the performance of the RFF preconditioner were similar to the performance of the algorithm without preconditioner.
This might have been that when taking small values of $s$ the RFF do not approximate the kernel well and leads to a larger condition number.

\begin{figure}[h!]
	\centering
	\begin{subfigure}[b]{0.45\linewidth}
		\includegraphics[width=\linewidth,height=7cm]{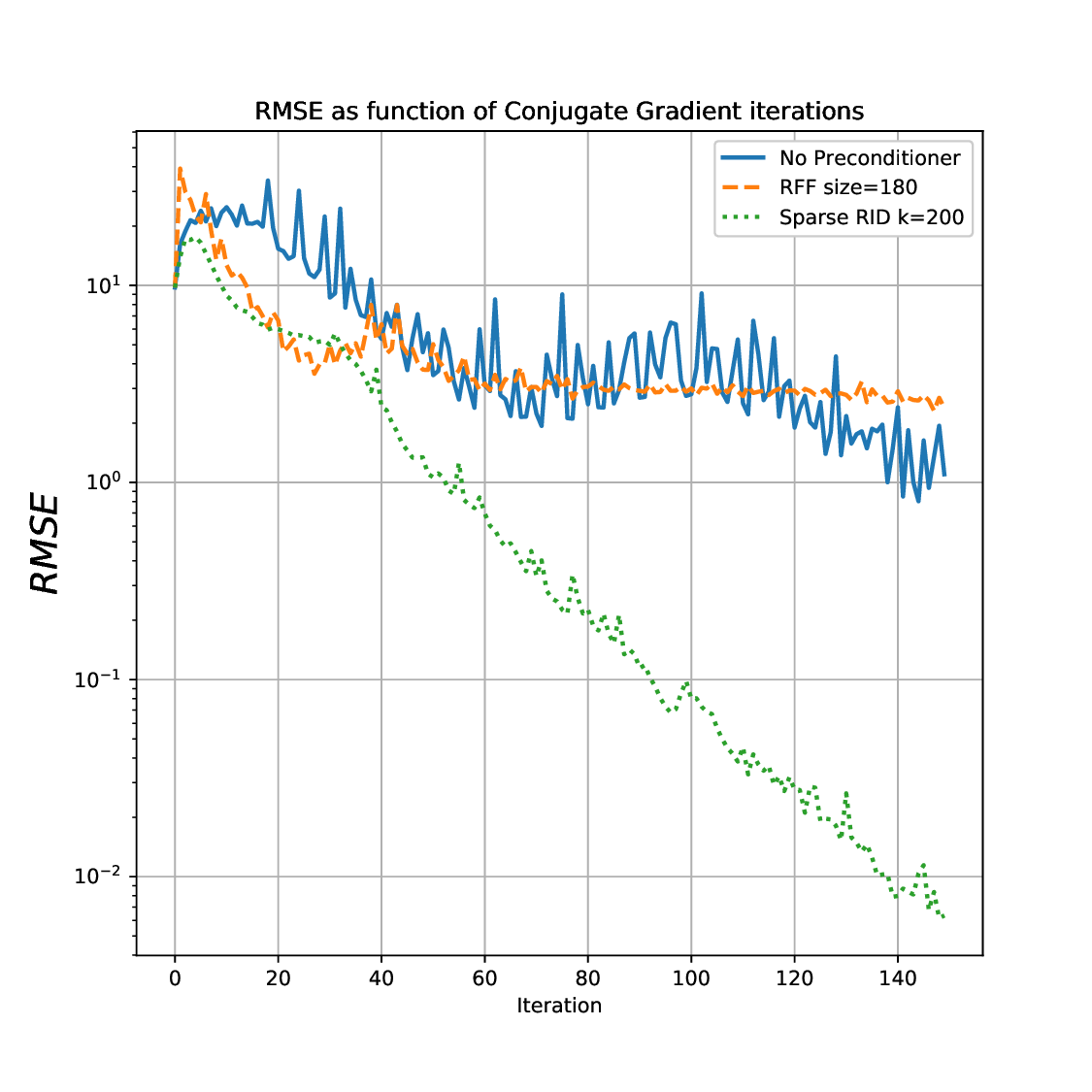}
		\caption{Error, standard norm}
	\end{subfigure}
	\begin{subfigure}[b]{0.45\linewidth}
		\includegraphics[width=\linewidth,height=7cm]{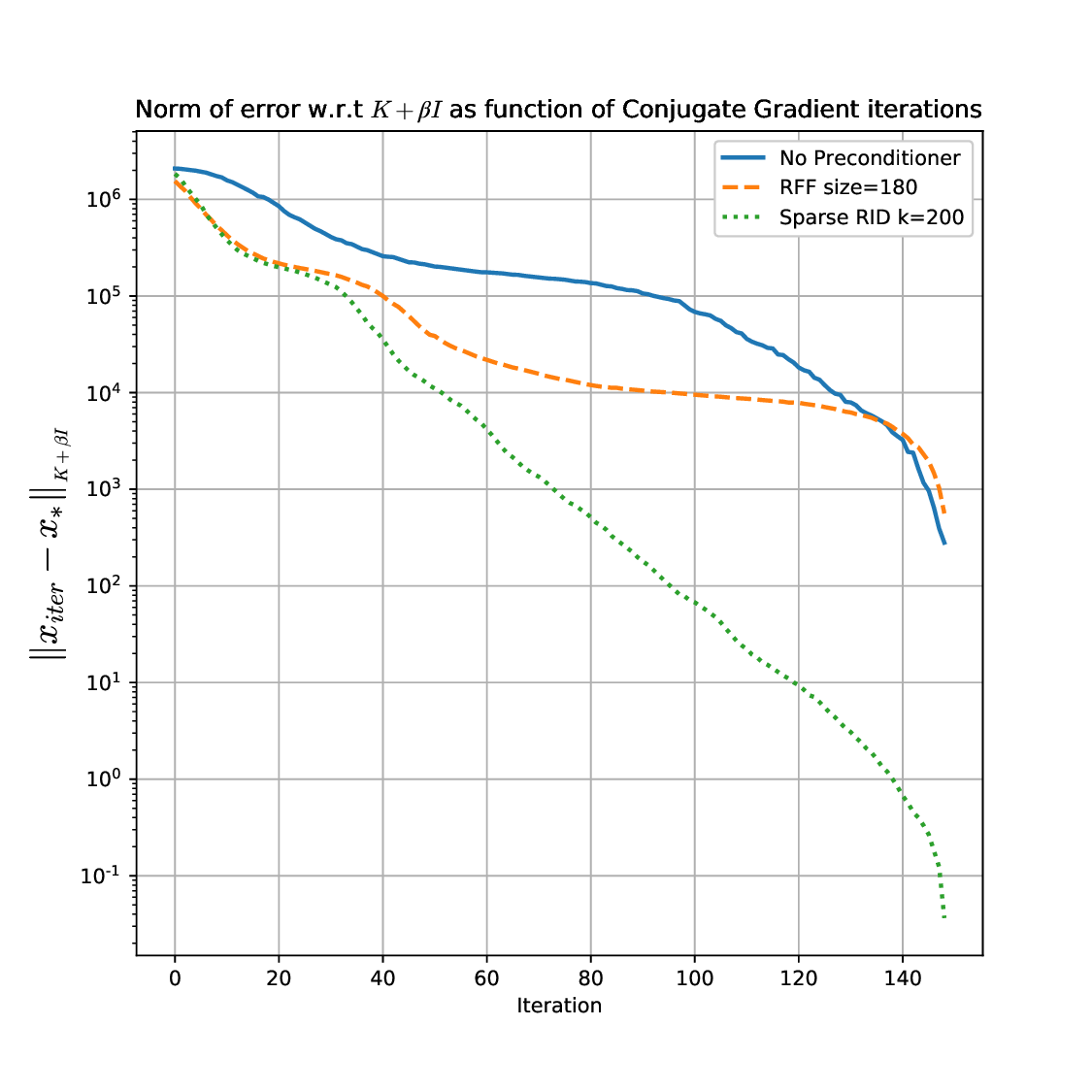}
		\caption{Error, inner-product norm}
	\end{subfigure}
	\caption{100K datapoints of the EM dataset, $k=200,$RFF size=180$, h=0.1,~\beta=0.1$}
	\label{fig:EM_RFF30_100k}
\end{figure}

\begin{figure}[h!]
	\centering
	\begin{subfigure}[b]{0.45\linewidth}
		\includegraphics[width=\linewidth,height=7cm]{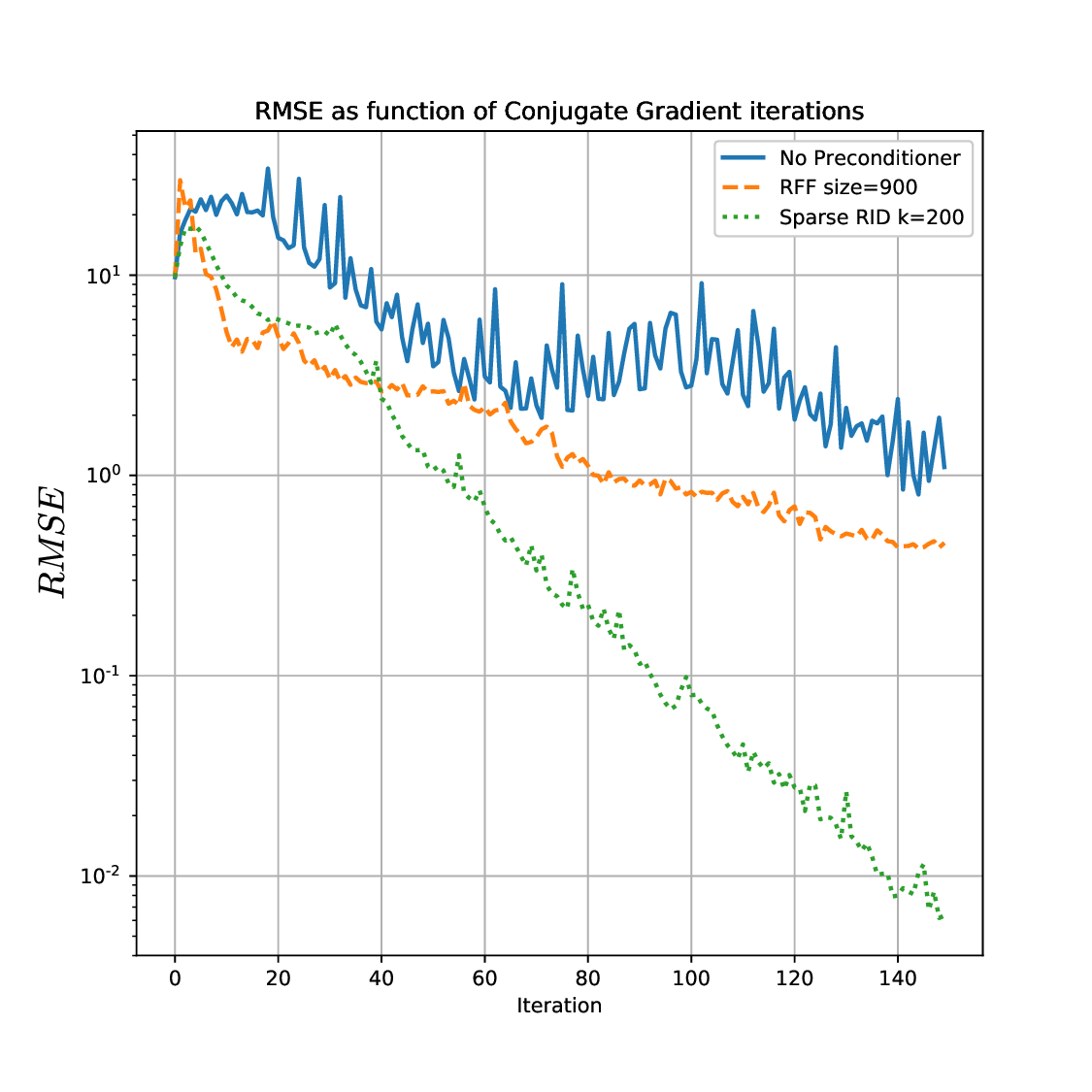}
		\caption{Error, standard norm}
	\end{subfigure}
	\begin{subfigure}[b]{0.45\linewidth}
		\includegraphics[width=\linewidth,height=7cm]{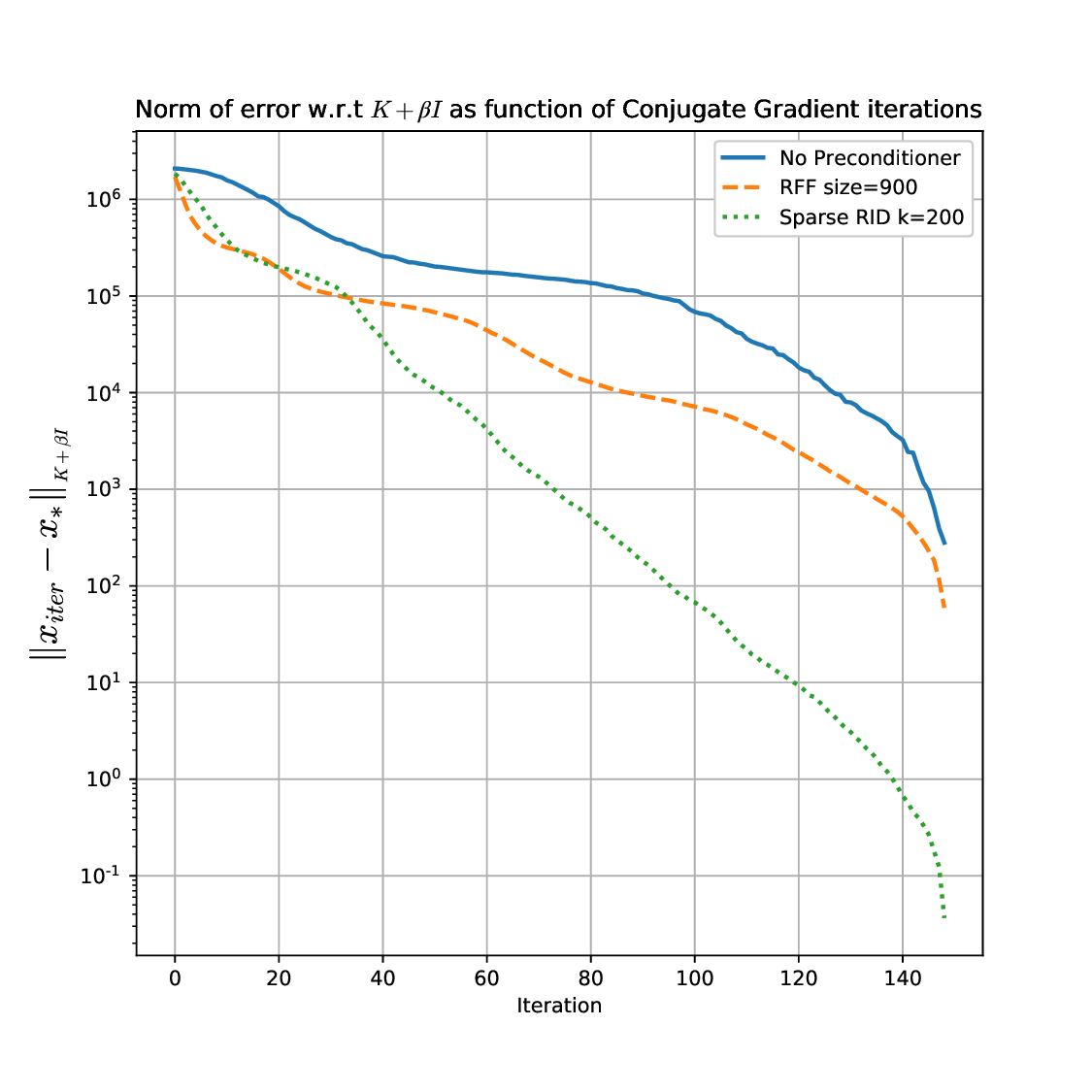}
		\caption{Error, inner-product norm}
	\end{subfigure}
	\caption{100K datapoints of the EM dataset, $k=200,$ RFF size=900$, h=0.1, \beta=0.1$}
	\label{fig:EM_RFF150_100k}
\end{figure}
\subsection{Results on High Dimensional Datasets}
The method was tested on larger datasets: \texttt{CoverType}, \texttt{Buzz} and \texttt{YearSong}. 
Since the method uses the improved FGT, the method works well even for relatively high dimensions such as $d=77$ and $d=90$. \texttt{CoverType} is a classification (multiclass) dataset, where the label has 7 options. Therefore, the algorithm was applied 7 times, using $+1$ for $i$th type and $-1$ for the others ($i=1,\ldots,7$). The predicted class was determined by the highest score, leading to $90\%$ accuracy on the test data. As for the regression datasets, the algorithm reached a stopping point for $\Vert Ax-b \Vert_2 / \Vert b \Vert_2 \le 10^{-3}$. 
\subsubsection{Comparison to Random Fourier Features Preconditioner}
First, the performance of the algorithm was tested on the \texttt{YearSong} dataset. The dataset is of size 515K in total and the dimension is $90$. For the training, $2/3$ of the dataset was used. The features were normalized to have zero mean and unit variance. The parameters were $h=9$ and $\beta=0.37$. Since in the comparison with the RFF preconditioner, the exact same problems are solved, just with different preconditioner, the important thing to compare is the convergence rate, rather than the error on the test dataset. Figure \ref{fig:YearMSD_Conv} shows the convergence rate of Algorithm \ref{alg:SolveKRR} compared to the RFF preconditioner and no preconditioner. For the RFF we used $s=900$ and for Algortihm \ref{alg:SolveKRR} we used $k=495$ and $l=500$.
\begin{figure}[H]
	\centering
	\begin{subfigure}[b]{0.45\linewidth}
		\includegraphics[width=\linewidth,height=7cm]{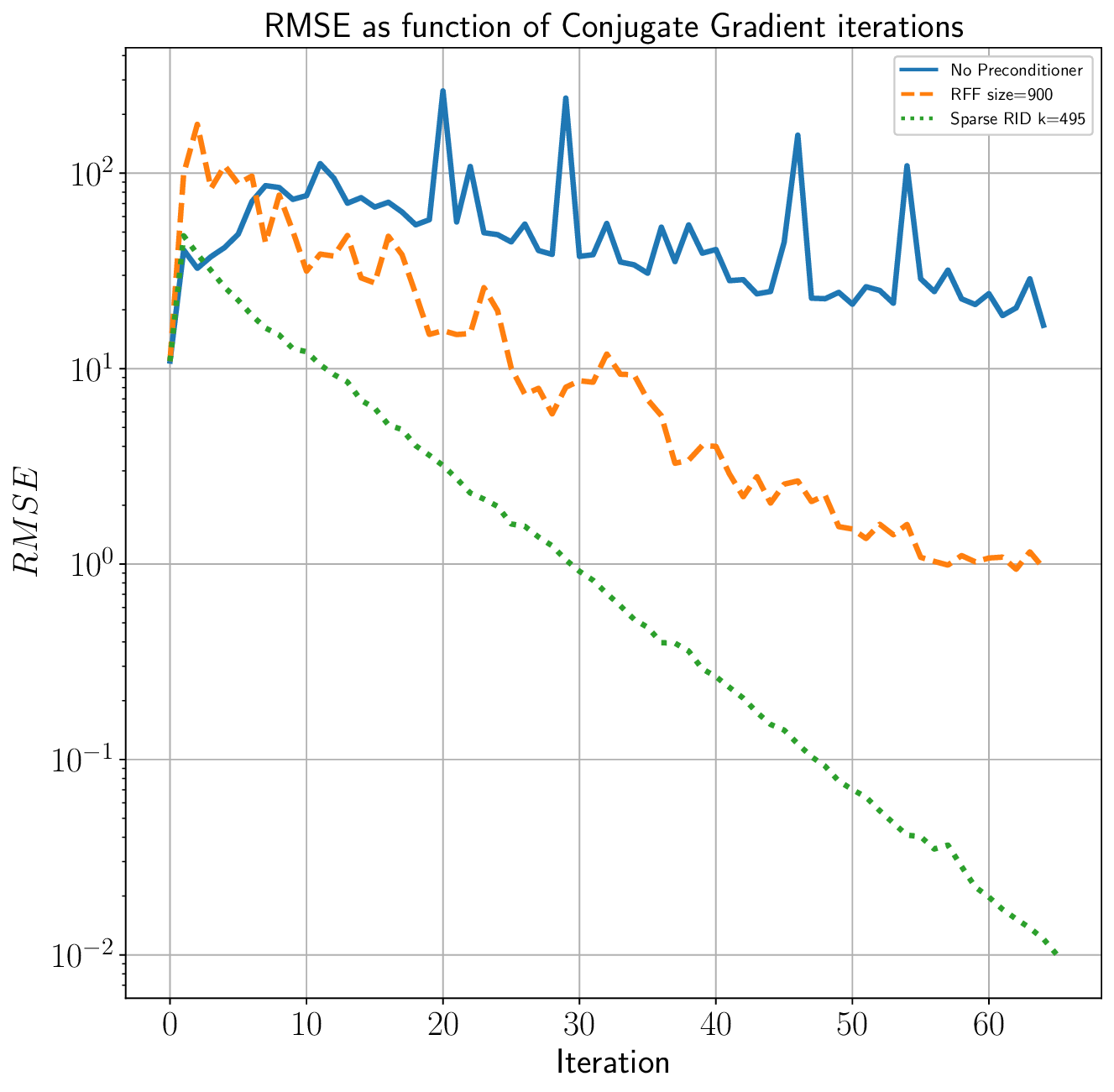}
		\caption{RMSE}
	\end{subfigure}
	\begin{subfigure}[b]{0.45\linewidth}
		\includegraphics[width=\linewidth,height=7cm]{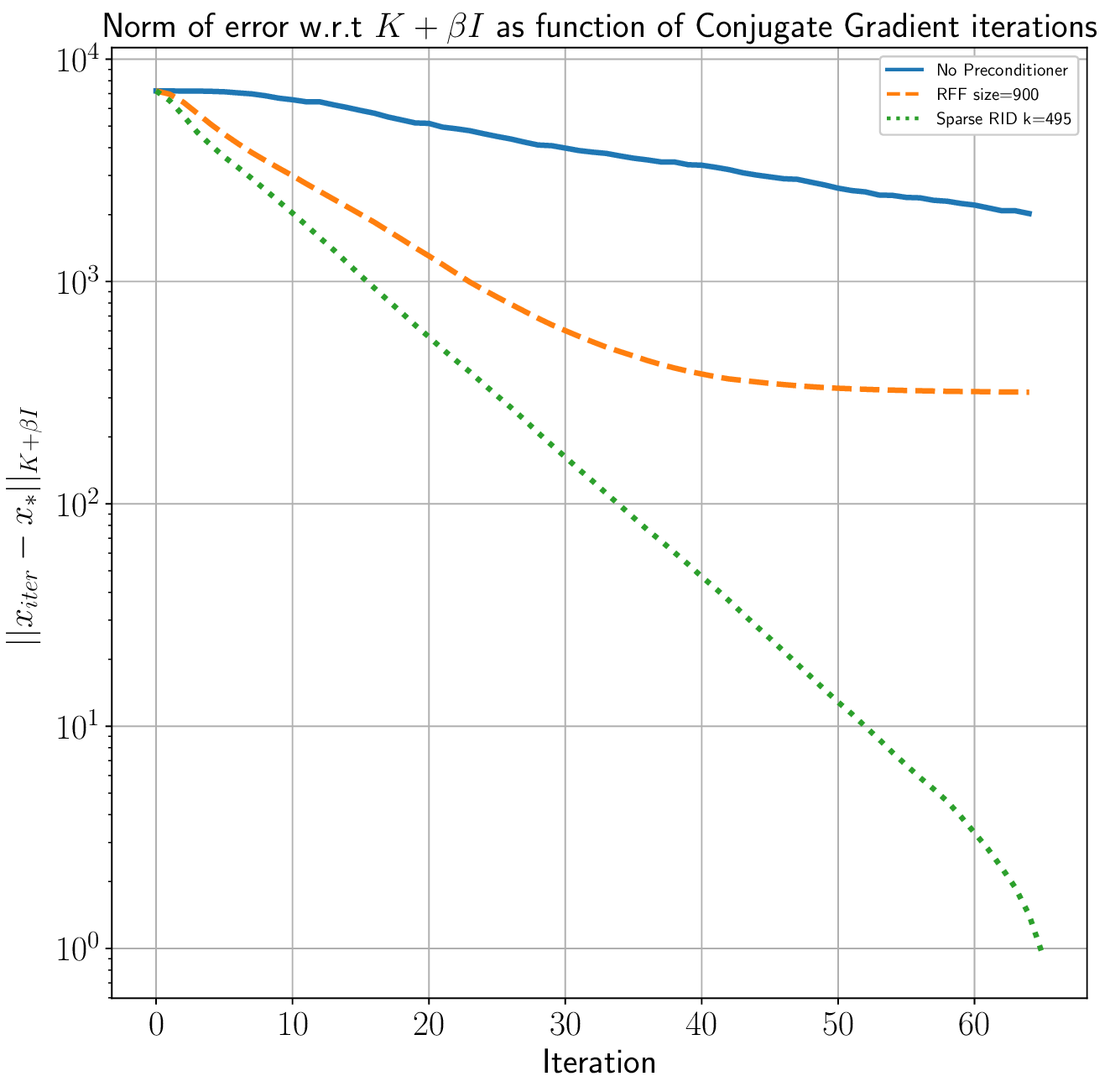}
		\caption{Inner-product norm error}
	\end{subfigure}
	\caption{Convergence comparison for the \texttt{YearSong} training dataset $h=9, \beta=0.37$}
	\label{fig:YearMSD_Conv}
\end{figure}
The experiment was repeated on the \texttt{CoverType} dataset. The dataset has $581K$ measurements and the dimension is $54$. The dataset was normalized to have zero mean and unit variance and $2/3$ of the data was used for training. The parameters used were $h=2, \beta=0.1, k=795, l=800$ and $s=648$. 
Figure \ref{fig:Covtype_Conv} shows the comparison between the suggested preconditioner, the RFF and no preconditioner. 

\begin{figure}[H]
	\centering
	\begin{subfigure}[b]{0.45\linewidth}
		\includegraphics[width=\linewidth,height=7cm]{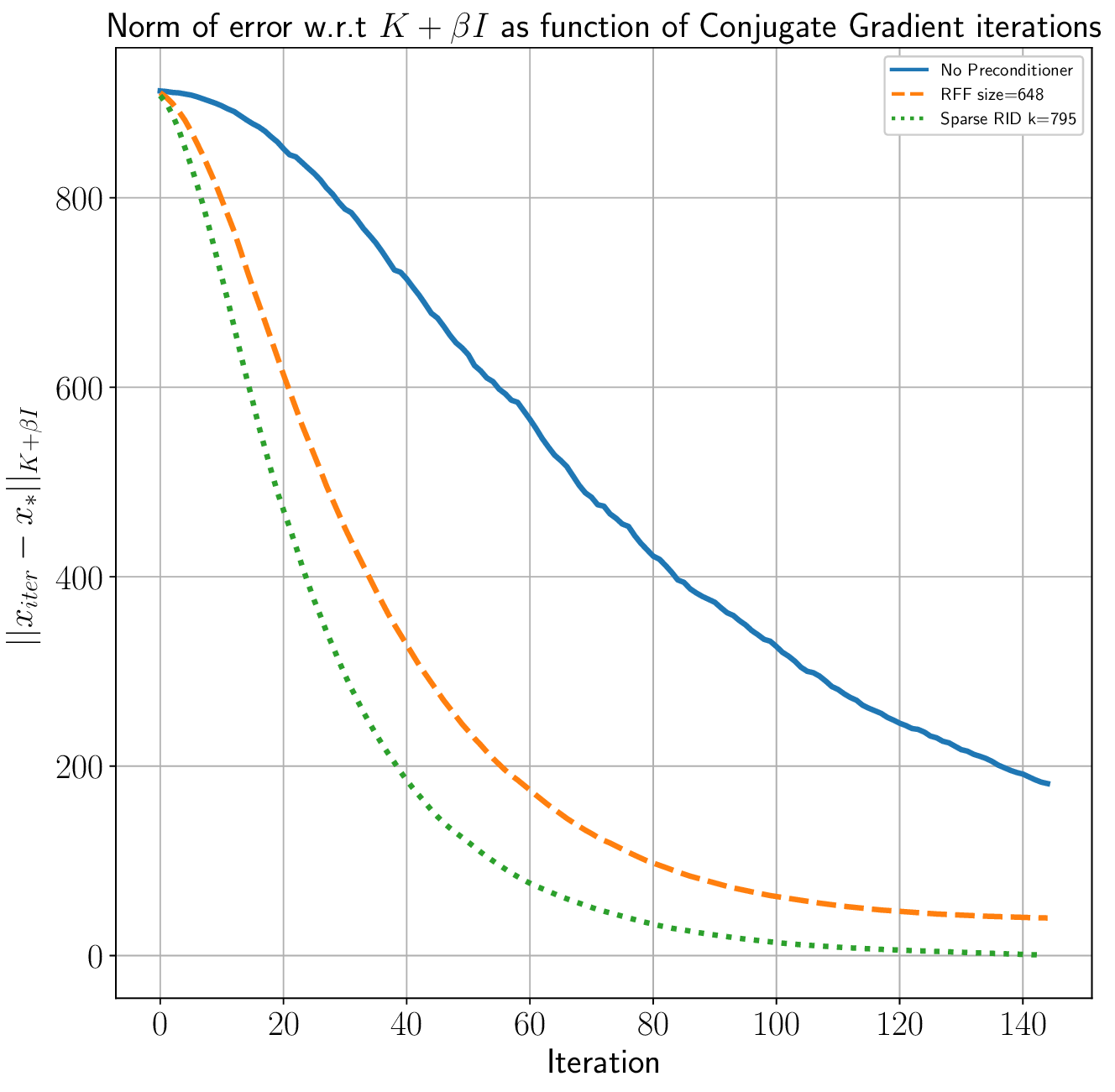}
		\caption{RMSE}
	\end{subfigure}
	\begin{subfigure}[b]{0.45\linewidth}
		\includegraphics[width=\linewidth,height=7cm]{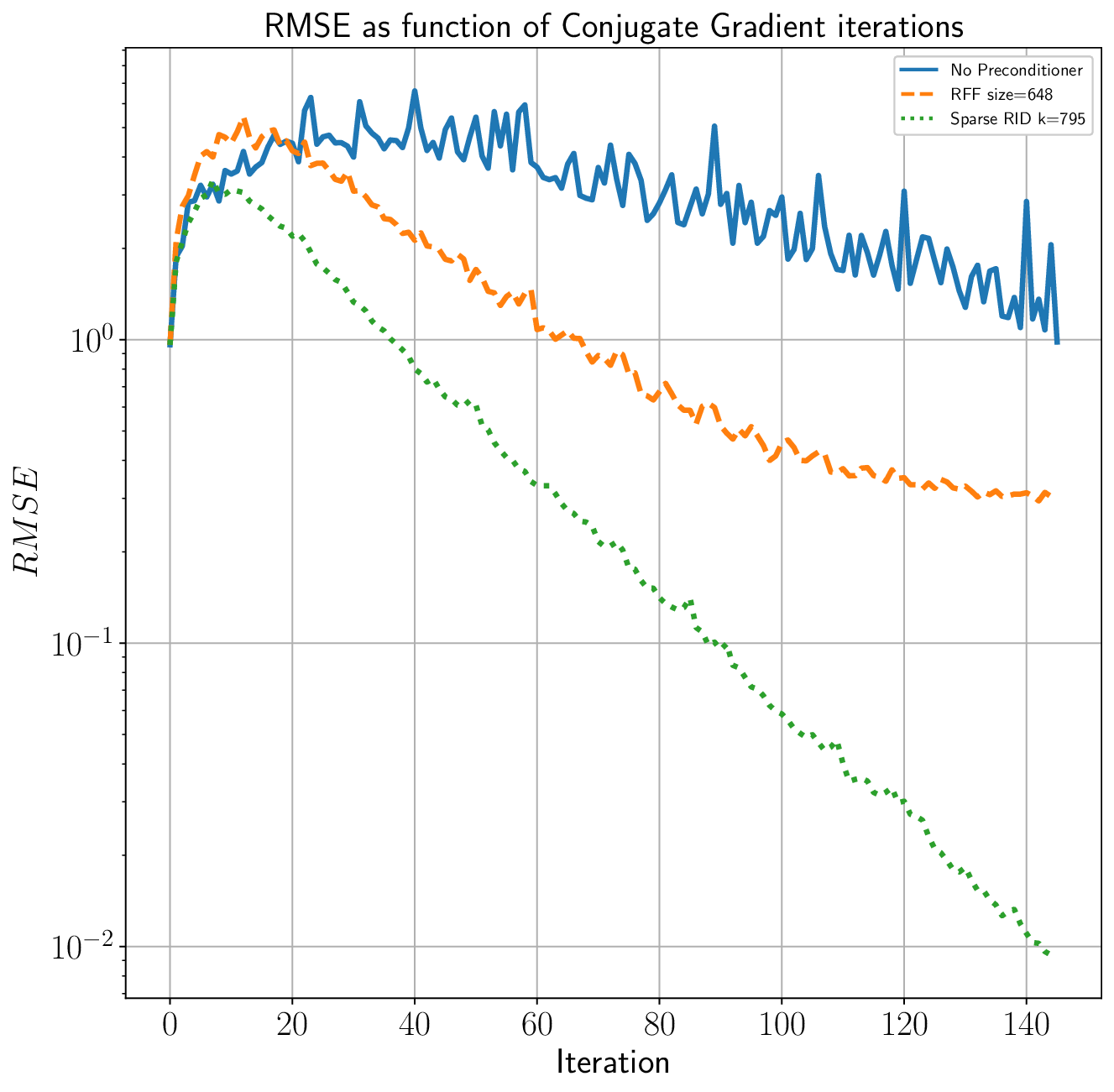}
		\caption{Inner-product norm error}
	\end{subfigure}
	\caption{Convergence comparison for the \texttt{CovType} training dataset $h=2, \beta=0.1$}
	\label{fig:Covtype_Conv}
\end{figure}

Next, the same experiment is conducted on the \texttt{Buzz} dataset. The size of the dataset is $583K$ and the dimension is $77$.  
The dataset was normalized to have a minimal value of zero and a maximal value of one. Again $2/3$ of the dataset was used for training. The parameters were $h=1, \beta=1, k=795, l=800$ and $s=770$. Figure \ref{fig:Buzz_Conv} shows the comparison between the suggested preconditioner, the RFF and no preconditioner. On this dataset, the performance of both preconditioners are similar. The convergence in RMSE is similar, and the convergence in the norm is slightly better with the RFF, where it has an advantage of about 4 iterations along the way. However, it seems to reach a plateau towards the end (last 5 iterations), while the KRR seems to keep the same slope.

\begin{figure}[H]
	\centering
	\begin{subfigure}[b]{0.45\linewidth}
		\includegraphics[width=\linewidth,height=7cm]{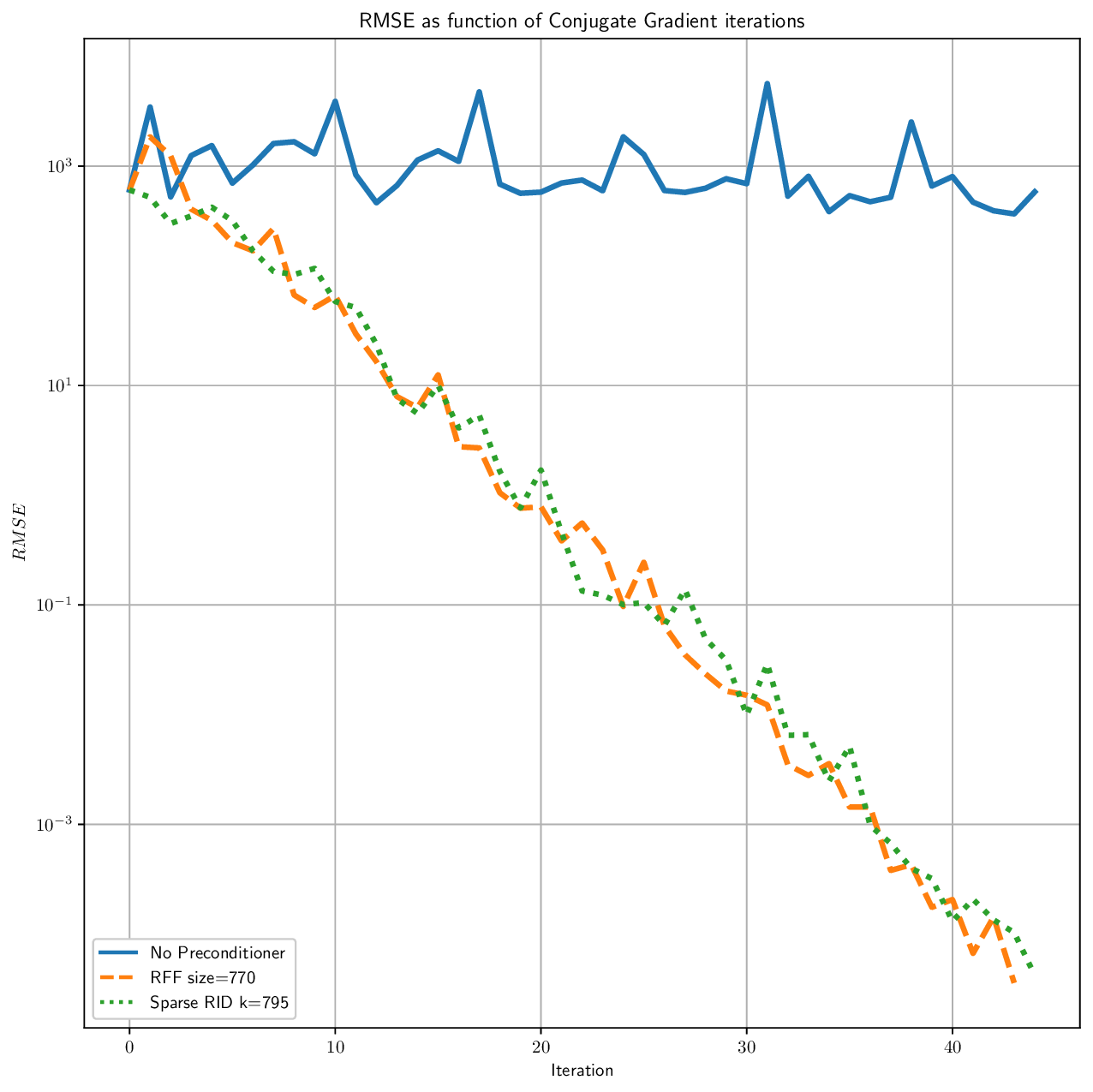}
		\caption{RMSE}
	\end{subfigure}
	\begin{subfigure}[b]{0.45\linewidth}
		\includegraphics[width=\linewidth,height=7cm]{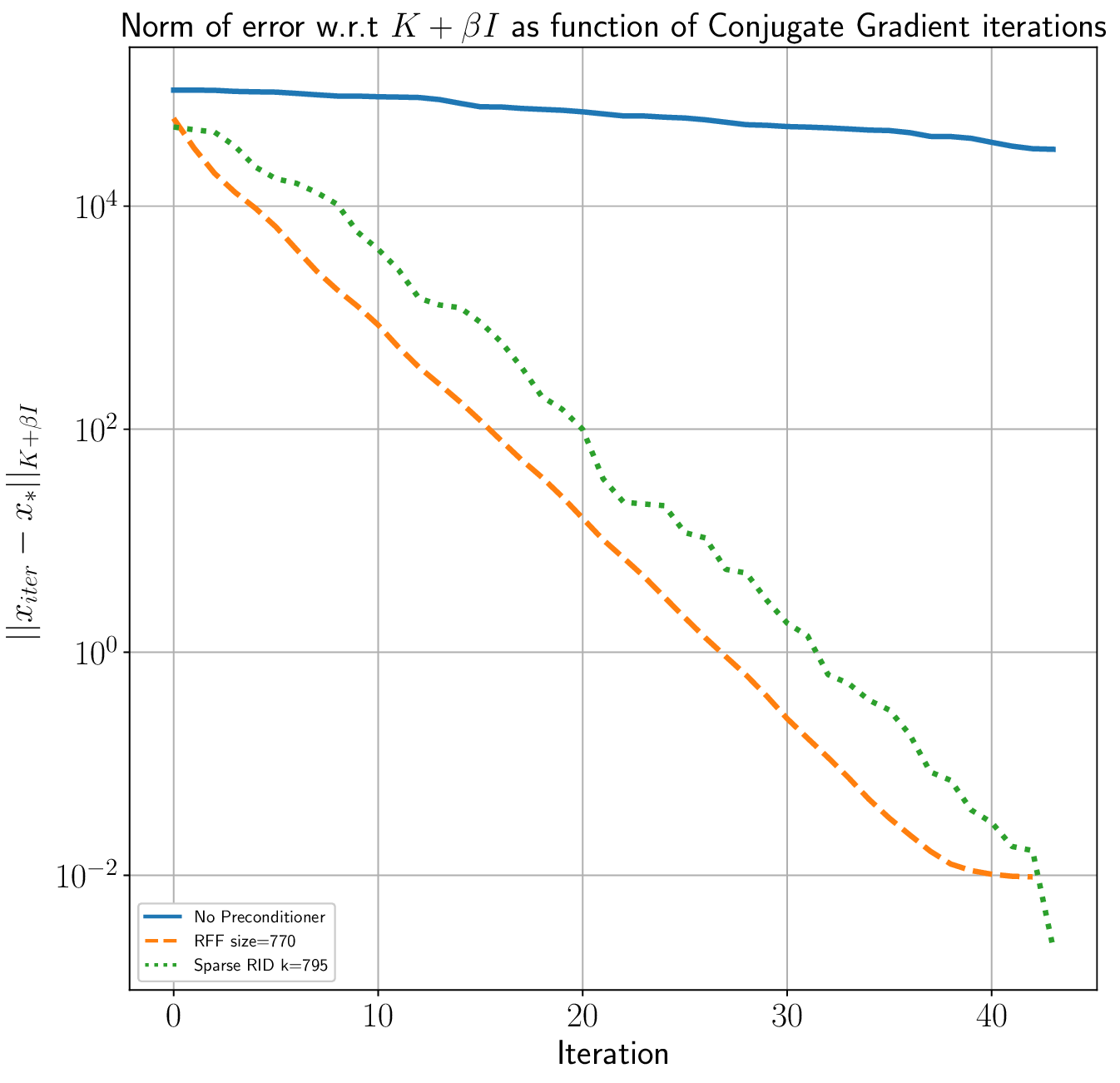}
		\caption{Inner-product norm error}
	\end{subfigure}
	\caption{Convergence comparison for the \texttt{Buzz} training dataset $h=1, \beta=1$}
	\label{fig:Buzz_Conv}
\end{figure}

\subsubsection{Comparison with the FALKON algorithm}
We now perform another comparison, this time with the FALKON algorithm \cite{rudi2017falkon}. The FALKON is a very fast algorithm that can solve KRR on very large datasets by using stochastic subsampling and preconditioning. In general, the FALKON samples $m$ points out of $n$, requires $\mathcal{O}({m^3})$ operations for the preprocessing and $\mathcal{O}(mn)$ operations for each iterations. According to \cite{rudi2017falkon}, it seems that in order to achieve high accuracy, 
it is usually sufficient to sample $m=\sqrt{n}$ points. For the FALKON algorithm, comparison is not as simple as for  RFF, since the algorithms are not as similar. We tried to compare in a way such that the FALKON will get a higher computational budget (by taking a $m > \sqrt{n}$) and comparing the error to the nearly exact KRR solution presented in the paper.
The number of iterations for the FALKON was not limited so the algorithm achieves full convergence. We repeat the experiments performed on the high dimensional datasets above.

\begin{table}
	\resizebox{\columnwidth}{!}{%
		\begin{tabular}{llllllllll}
			Dataset    & Train Size& d & $h$  & $\beta$  & Method & Parameter ($m,l$) & Iterations & Train RMSE & Test RMSE \\\specialrule{3pt}{0pt}{0pt}
			Buzz       & 390777                    & 77        & 1   & 1 & Algorithm \ref{alg:SolveKRR}   & $l=800$  & 45                  & 145    & 207     \\
			Buzz       & 390777                    & 77        & 1   & 1 & FALKON                   & $m=20000$  & 33                 & 200  & 257     \\
			Buzz       & 390777                    & 77        & 1   & 1 & RFF                      & $s=770$  & 44                  & 145  & 207     \\
			Buzz       & 390777                    & 77        & 1   & 1 & No precond               & N/A  & 45                  & 504  & 524     \\
			YearMSD    & 345281                    & 90        & 9  & 0.37 & Algorithm \ref{alg:SolveKRR}   & $l=500$ & 65                  & 6.79     & 8.49    \\
			YearMSD    & 345281                    & 90        & 9  & 0.37 & FALKON                         & $m=8000$ & 37                  & 8.65     & 8.81    \\
			YearMSD    & 345281                    & 90        & 9  & 0.37 & RFF                            & $s=900$ & 65                  & 6.86     & 8.54    \\
			YearMSD    & 345281                    & 90        & 9  & 0.37 & No precond                         & N/A & 65                  & 26     & 26.4    \\
			CovType & 389278                    & 54        & 2 & 0.1 & Algorithm \ref{alg:SolveKRR}   & $l=800$           & 145                  & 0.43   & 0.46    \\
			CovType & 389278                    & 54        & 2 & 0.1 & FALKON   & $m=15000$           & 36                 & 0.51   & 0.55    \\
			CovType & 389278                    & 54        & 2 & 0.1 & RFF   & $s=648$           & 145                  & 0.52   & 0.55    \\
			CovType & 389278                    & 54        & 2 & 0.1 & No precond   & N/A            & 145                  & 1.01   & 1.02    \\
		\end{tabular}
	}	
	\caption{Results for a variety of large datasets}
	\label{table:falkon_compare}
\end{table}

Therefore, we added a comparison between the fraction of preprocessing time needed for those datasets with respect to the FALKON and the RFF. The preprocessing time is measured as a fraction of the total running time. All the parameters are as appear in Table \ref{table:falkon_compare}. 
The results of fraction of time required for preprocessing are in Table \ref{table:fraction_time}, with parameters identical to those in Table \ref{table:falkon_compare}.
Indeed, the suggested KRR algorithm has the highest preprocessing time compared to the other algorithms, though it is still relatively small with respect to the iterations part.

\begin{table}[H]
	\centering
	\begin{tabular}{lll}
		Dataset & ~~~~~~Method 							 &	~~~~~~Preprocessing time fraction  
		 \\\specialrule{3pt}{0pt}{0pt}
		YearMSD &  ~~~~Algorithm \ref{alg:SolveKRR}      &      ~~~~~~~~~~~~~~~$3.8\%$     \\
		YearMSD &  ~~~~FALKON 						     &      ~~~~~~~~~~~~~~~$1.5\%$                            \\
		YearMSD &  ~~~~RFF    						     &      ~~~~~~~~~~~~~~~$0.026\%$                 \\          
		CovType &  ~~~~Algorithm \ref{alg:SolveKRR}      &      ~~~~~~~~~~~~~~~$2.7\%$                       \\
		CovType &  ~~~~FALKON 						     &      ~~~~~~~~~~~~~~~$3\%$                            \\
		CovType &  ~~~~RFF    						     &      ~~~~~~~~~~~~~~~$0.016\%$                 \\ 
		Buzz &  ~~~~Algorithm \ref{alg:SolveKRR}      &      ~~~~~~~~~~~~~~~$6.7\%$                       \\
		Buzz &  ~~~~FALKON 						     &      ~~~~~~~~~~~~~~~$5.4\%$                            \\
		Buzz &  ~~~~RFF    						     &      ~~~~~~~~~~~~~~~$0.034\%$                              
	\end{tabular}
\caption{The percentage of preprocessing time for different algorithms and datasets}
\label{table:fraction_time}

\end{table}


\section{Conclusions}
Kernel ridge regression is a powerful technique with a theoretic basis that is often constrained by memory and run-time complexity when used naively. In this paper we narrowed our focus to Gaussian kernels in order to utilize fast MVM operations while maintaining the same specialized structure for the preconditioner as well, which allowed us to maintain the speed-up of fast MVM operations in the application of the preconditioned kernel. 
One major drawback of our method is the polynomial dependency on the dimension of the data, which despite improvement in the fast improved Gauss transform, still hinders the application of the algorithm to extremely high dimensional datasets. However, the FIG transform is used as a black-box and can be replaced by other MVM or even in some cases by direct computation of the full matrix (probably in parts). Also, the preconditioner can be used ``as is" for other kernels, with or without utilizing fast MVM.
In addition, we provide theoretical insights and performance bounds to our method, which build on the rich existing theoretical work on kernel ridge regression and numerical linear algebra. As shown, our method provides a way to choose the required low-rank approximation with the appropriate tradeoff in speed.
We presented an algorithm that is strict in memory, exact, allows efficient application of the preconditioner and reduces the condition number. Our results show fast convergence of the CG algorithm, outperforming similar SOTA methods. In particular, the algorithm scales well for large values of n, and is suitable for cases where $n \gg d$.

\clearpage
\bibliography{Faster_KRR_SIAM_R2}

\begin{thebibliography}{10}

\bibitem{helicopter}
{\sc P.~Abbeel, A.~Coates, and A.~Y. Ng}, {\em Autonomous helicopter aerobatics
  through apprenticeship learning}, The International Journal of Robotics
  Research, 29 (2010), pp.~1608--1639.

\bibitem{aizenbud2016randomized}
{\sc Y.~Aizenbud, G.~Shabat, and A.~Averbuch}, {\em Randomized {LU}
  decomposition using sparse projections}, Computers \& Mathematics with
  Applications, 72 (2016), pp.~2525--2534.

\bibitem{alfke2018nfft}
{\sc D.~Alfke, D.~Potts, M.~Stoll, and T.~Volkmer}, {\em Nfft meets krylov
  methods: Fast matrix-vector products for the graph laplacian of fully
  connected networks}, Frontiers in Applied Mathematics and Statistics, 4
  (2018), p.~61.

\bibitem{avron_krr}
{\sc H.~Avron, K.~L. Clarkson, and D.~P. Woodruff}, {\em Faster kernel ridge
  regression using sketching and preconditioning}, SIAM Journal on Matrix
  Analysis and Applications, 38 (2017), pp.~1116--1138.

\bibitem{avron2016high}
{\sc H.~Avron and V.~Sindhwani}, {\em High-performance kernel machines with
  implicit distributed optimization and randomization}, Technometrics, 58
  (2016), pp.~341--349.

\bibitem{bermanis2013multiscale}
{\sc A.~Bermanis, A.~Averbuch, and R.~R. Coifman}, {\em Multiscale data
  sampling and function extension}, Applied and Computational Harmonic
  Analysis, 34 (2013), pp.~15--29.

\bibitem{bermanis2014cover}
{\sc A.~Bermanis, G.~Wolf, and A.~Averbuch}, {\em Cover-based bounds on the
  numerical rank of {G}aussian kernels}, Applied and Computational Harmonic
  Analysis, 36 (2014), pp.~302--315.

\bibitem{bernstein2009matrix}
{\sc D.~S. Bernstein}, {\em Matrix Mathematics: Theory, Facts, and Formulas},
  Princeton University Press, 2009.

\bibitem{bronstein2008numerical}
{\sc A.~M. Bronstein, M.~M. Bronstein, and R.~Kimmel}, {\em Numerical geometry
  of non-rigid shapes}, Springer Science \& Business Media, 2008.

\bibitem{chandrasekaran1994rank}
{\sc S.~Chandrasekaran and I.~C. Ipsen}, {\em On rank-revealing
  factorisations}, SIAM Journal on Matrix Analysis and Applications, 15 (1994),
  pp.~592--622.

\bibitem{cheng2005compression}
{\sc H.~Cheng, Z.~Gimbutas, P.-G. Martinsson, and V.~Rokhlin}, {\em On the
  compression of low rank matrices}, SIAM Journal on Scientific Computing, 26
  (2005), pp.~1389--1404.

\bibitem{clarkson2017low}
{\sc K.~L. Clarkson and D.~P. Woodruff}, {\em Low-rank approximation and
  regression in input sparsity time}, Journal of the ACM (JACM), 63 (2017),
  p.~54.

\bibitem{cutajar2016preconditioning}
{\sc K.~Cutajar, M.~Osborne, J.~Cunningham, and M.~Filippone}, {\em
  Preconditioning kernel matrices}, in International Conference on Machine
  Learning, 2016, pp.~2529--2538.

\bibitem{dai2014scalable}
{\sc B.~Dai, B.~Xie, N.~He, Y.~Liang, A.~Raj, M.-F.~F. Balcan, and L.~Song},
  {\em Scalable kernel methods via doubly stochastic gradients}, in Advances in
  Neural Information Processing Systems, 2014, pp.~3041--3049.

\bibitem{dasgupta2010sparse}
{\sc A.~Dasgupta, R.~Kumar, and T.~Sarl{\'o}s}, {\em A sparse
  {J}ohnson--{L}indenstrauss transform}, in Proceedings of the forty-second
  {ACM} symposium on Theory of computing, ACM, 2010, pp.~341--350.

\bibitem{drineas2005nystrom}
{\sc P.~Drineas and M.~W. Mahoney}, {\em On the {N}ystr{\"o}m method for
  approximating a gram matrix for improved kernel-based learning}, journal of
  machine learning research, 6 (2005), pp.~2153--2175.

\bibitem{drineas2008relative}
{\sc P.~Drineas, M.~W. Mahoney, and S.~Muthukrishnan}, {\em Relative-error
  {CUR} matrix decompositions}, SIAM Journal on Matrix Analysis and
  Applications, 30 (2008), pp.~844--881.

\bibitem{freitas2006fast}
{\sc N.~D. Freitas, Y.~Wang, M.~Mahdaviani, and D.~Lang}, {\em Fast {K}rylov
  methods for {N}-body learning}, in Advances in Neural Information Processing
  Systems, 2006, pp.~251--258.

\bibitem{gardner2018gpytorch}
{\sc J.~Gardner, G.~Pleiss, K.~Q. Weinberger, D.~Bindel, and A.~G. Wilson},
  {\em Gpytorch: Blackbox matrix-matrix gaussian process inference with gpu
  acceleration}, in Advances in Neural Information Processing Systems, 2018,
  pp.~7576--7586.

\bibitem{golub2012matrix}
{\sc G.~H. Golub and C.~F. Van~Loan}, {\em Matrix computations}, vol.~4, John
  Hopkins University Press, 2012.

\bibitem{goreinov1997theory}
{\sc S.~A. Goreinov, E.~E. Tyrtyshnikov, and N.~L. Zamarashkin}, {\em A theory
  of pseudoskeleton approximations}, Linear algebra and its applications, 261
  (1997), pp.~1--21.

\bibitem{greengard1987fast}
{\sc L.~Greengard and V.~Rokhlin}, {\em A fast algorithm for particle
  simulations}, Journal of computational physics, 73 (1987), pp.~325--348.

\bibitem{greengard1991fast}
{\sc L.~Greengard and J.~Strain}, {\em The fast {G}auss transform}, SIAM
  Journal on Scientific and Statistical Computing, 12 (1991), pp.~79--94.

\bibitem{gu1996efficient}
{\sc M.~Gu and S.~C. Eisenstat}, {\em Efficient algorithms for computing a
  strong rank-revealing {QR} factorization}, SIAM Journal on Scientific
  Computing, 17 (1996), pp.~848--869.

\bibitem{halko2011finding}
{\sc N.~Halko, P.-G. Martinsson, and J.~A. Tropp}, {\em Finding structure with
  randomness: Probabilistic algorithms for constructing approximate matrix
  decompositions}, SIAM review, 53 (2011), pp.~217--288.

\bibitem{RRQR_svd}
{\sc Y.~Hong and C.-T. Pan}, {\em Rank-revealing {QR} factorizations and the
  singular value decomposition}, Math. Comp, 58 (1992), pp.~213--232.

\bibitem{liberty2007randomized}
{\sc E.~Liberty, F.~Woolfe, P.-G. Martinsson, V.~Rokhlin, and M.~Tygert}, {\em
  Randomized algorithms for the low-rank approximation of matrices},
  Proceedings of the National Academy of Sciences, 104 (2007),
  pp.~20167--20172.

\bibitem{lu2014scale}
{\sc Z.~Lu, A.~May, K.~Liu, A.~B. Garakani, D.~Guo, A.~Bellet, L.~Fan,
  M.~Collins, B.~Kingsbury, M.~Picheny, et~al.}, {\em How to scale up kernel
  methods to be as good as deep neural nets}, arXiv preprint arXiv:1411.4000,
  (2014).

\bibitem{ma2017diving}
{\sc S.~Ma and M.~Belkin}, {\em Diving into the shallows: a computational
  perspective on large-scale shallow learning}, in Advances in Neural
  Information Processing Systems, 2017, pp.~3778--3787.

\bibitem{mahoney2009cur}
{\sc M.~W. Mahoney and P.~Drineas}, {\em {CUR} matrix decompositions for
  improved data analysis}, Proceedings of the National Academy of Sciences, 106
  (2009), pp.~697--702.

\bibitem{mahoney2011randomized}
{\sc M.~W. Mahoney et~al.}, {\em Randomized algorithms for matrices and data},
  Foundations and Trends{\textregistered} in Machine Learning, 3 (2011),
  pp.~123--224.

\bibitem{martinsson2011randomized}
{\sc P.-G. Martinsson, V.~Rokhlin, and M.~Tygert}, {\em A randomized algorithm
  for the decomposition of matrices}, Applied and Computational Harmonic
  Analysis, 30 (2011), pp.~47--68.

\bibitem{miranian2003strong}
{\sc L.~Miranian and M.~Gu}, {\em Strong rank revealing {LU} factorizations},
  Linear algebra and its applications, 367 (2003), pp.~1--16.

\bibitem{morariu2009automatic}
{\sc V.~I. Morariu, B.~V. Srinivasan, V.~C. Raykar, R.~Duraiswami, and L.~S.
  Davis}, {\em Automatic online tuning for fast {G}aussian summation}, in
  Advances in Neural Information Processing Systems, 2009, pp.~1113--1120.

\bibitem{musco2017recursive}
{\sc C.~Musco and C.~Musco}, {\em Recursive sampling for the {N}ystrom method},
  in Advances in Neural Information Processing Systems, 2017, pp.~3833--3845.

\bibitem{nemtsov2016matrix}
{\sc A.~Nemtsov, A.~Averbuch, and A.~Schclar}, {\em Matrix compression using
  the {N}ystr{\"o}m method}, Intelligent Data Analysis, 20 (2016),
  pp.~997--1019.

\bibitem{osinsky2018pseudo}
{\sc A.~Osinsky and N.~Zamarashkin}, {\em Pseudo-skeleton approximations with
  better accuracy estimates}, Linear Algebra and its Applications, 537 (2018),
  pp.~221--249.

\bibitem{pan2000existence}
{\sc C.-T. Pan}, {\em On the existence and computation of rank-revealing {LU}
  factorizations}, Linear Algebra and its Applications, 316 (2000),
  pp.~199--222.

\bibitem{rahimi_recht}
{\sc A.~Rahimi and B.~Recht}, {\em Random features for large-scale kernel
  machines}, Advances in Neural Information Processing Systems,  (2008),
  pp.~1177--1184.

\bibitem{raykar2007fast}
{\sc V.~C. Raykar and R.~Duraiswami}, {\em Fast large scale {G}aussian process
  regression using approximate matrix-vector products}, in Learning workshop,
  2007.

\bibitem{rudi2017falkon}
{\sc A.~Rudi, L.~Carratino, and L.~Rosasco}, {\em Falkon: An optimal large
  scale kernel method}, in Advances in Neural Information Processing Systems,
  2017, pp.~3888--3898.

\bibitem{shabat2018randomized}
{\sc G.~Shabat, Y.~Shmueli, Y.~Aizenbud, and A.~Averbuch}, {\em Randomized {LU}
  decomposition}, Applied and Computational Harmonic Analysis, 44 (2018),
  pp.~246--272.

\bibitem{shen2006fast}
{\sc Y.~Shen, M.~Seeger, and A.~Y. Ng}, {\em Fast {G}aussian process regression
  using kd-trees}, in Advances in Neural Information Processing Systems, 2006,
  pp.~1225--1232.

\bibitem{smola2000sparse}
{\sc A.~Smola and B.~Sch{\"o}lkopf}, {\em Sparse greedy matrix approximation
  for machine learning.}, in Seventeenth International Conference on Machine
  Learning (ICML 2000), Morgan Kaufmann, 2000, pp.~911--918.

\bibitem{srinivasan2014preconditioned}
{\sc B.~V. Srinivasan, Q.~Hu, N.~A. Gumerov, R.~Murtugudde, and R.~Duraiswami},
  {\em Preconditioned {K}rylov solvers for kernel regression}, arXiv preprint
  arXiv:1408.1237,  (2014).

\bibitem{voronin2017efficient}
{\sc S.~Voronin and P.-G. Martinsson}, {\em Efficient algorithms for {CUR} and
  interpolative matrix decompositions}, Advances in Computational Mathematics,
  43 (2017), pp.~495--516.

\bibitem{wang2019exact}
{\sc K.~A. Wang, G.~Pleiss, J.~R. Gardner, S.~Tyree, K.~Q. Weinberger, and
  A.~G. Wilson}, {\em Exact gaussian processes on a million data points}, arXiv
  preprint arXiv:1903.08114,  (2019).

\bibitem{wang2019block}
{\sc R.~Wang, Y.~Li, M.~W. Mahoney, and E.~Darve}, {\em Block basis
  factorization for scalable kernel evaluation}, SIAM Journal on Matrix
  Analysis and Applications, 40 (2019), pp.~1497--1526.

\bibitem{williams2001using}
{\sc C.~K. Williams and M.~Seeger}, {\em Using the {N}ystr{\"o}m method to
  speed up kernel machines}, in Advances in Neural Information Processing
  Systems, 2001, pp.~682--688.

\bibitem{yang2005efficient}
{\sc C.~Yang, R.~Duraiswami, and L.~S. Davis}, {\em Efficient kernel machines
  using the improved fast {G}auss transform}, in Advances in Neural Information
  Processing Systems, 2005, pp.~1561--1568.

\bibitem{yangimproved}
{\sc C.~Yang, R.~Duraiswami, N.~Gumerov, and L.~Davis}, {\em Improved fast
  {G}auss transform and efficient kernel density estimation}, in Proceedings
  Ninth IEEE International Conference on Computer Vision.

\end{thebibliography}
\bibliographystyle{siam}

\end{document}